\documentclass[11 pt]{amsart}
\usepackage{amsmath, amsthm,amssymb,bbm,enumerate,mathrsfs,mathtools}
\usepackage{a4wide}
\usepackage{tikz,xcolor,verbatim}
\usetikzlibrary{calc,shapes}
\usepackage{hyperref}
\usepackage{makecell}
\usepackage[numbers,sort&compress]{natbib}

\usepackage{amssymb}
\usepackage{algorithmicx}
\usepackage{algpseudocode}
\usepackage{algorithm}
\usepackage{a4wide}
\algblock{Input}{EndInput}
\algnotext{EndInput}
\algblock{Output}{EndOutput}
\algnotext{EndOutput}
\algblock{Initialization}{EndInitialization}
\algnotext{EndInitialization}
\algblock{Begin}{EndBegin}
\algnotext{EndBegin}

 \usepackage[yyyymmdd,hhmmss]{datetime}
 \usepackage{background}
 \SetBgContents{ver. \currenttime \qquad \today }
 \SetBgScale{1}
 \SetBgAngle{0}
 \SetBgOpacity{1}
 \SetBgPosition{current page.south west} 
 \SetBgHshift{3cm}
 \SetBgVshift{0.5cm} 

\tikzset{
  bigblue/.style={circle, draw=blue!80,fill=blue!40,thick, inner sep=1.5pt, minimum size=5mm},
  bigred/.style={circle, draw=red!80,fill=red!40,thick, inner sep=1.5pt, minimum size=5mm},
  bigblack/.style={circle, draw=black!100,fill=black!40,thick, inner sep=1.5pt, minimum size=5mm},
  bluevertex/.style={circle, draw=blue!100,fill=blue!100,thick, inner sep=0pt, minimum size=2mm},
  redvertex/.style={circle, draw=red!100,fill=red!100,thick, inner sep=0pt, minimum size=2mm},
  blackvertex/.style={circle, draw=black!100,fill=black!100,thick, inner sep=0pt, minimum size=1.5mm},  
  whitevertex/.style={circle, draw=black!100,fill=white!100,thick, inner sep=0pt, minimum size=2mm},  
  smallblack/.style={circle, draw=black!100,fill=black!100,thick, inner sep=0pt, minimum size=1mm},
  smallwhite/.style={circle, draw=black!100,fill=white!100,thick, inner sep=0pt, minimum size=1mm} 
}

\makeatletter
\pgfdeclareshape{myNode}{
  \inheritsavedanchors[from=rectangle] 
  \inheritanchorborder[from=rectangle]
  \inheritanchor[from=rectangle]{center}
  \inheritanchor[from=rectangle]{north}
  \inheritanchor[from=rectangle]{south}
  \inheritanchor[from=rectangle]{west}
  \inheritanchor[from=rectangle]{east}
  \backgroundpath{
    \southwest \pgf@xa=\pgf@x \pgf@ya=\pgf@y
    \northeast \pgf@xb=\pgf@x \pgf@yb=\pgf@y
    \pgfsetcornersarced{\pgfpoint{5pt}{5pt}}
    \pgfpathmoveto{\pgfpoint{\pgf@xa}{\pgf@ya}}
    \pgfpathlineto{\pgfpoint{\pgf@xa}{\pgf@yb}}
    \pgfpathlineto{\pgfpoint{\pgf@xb}{\pgf@yb}}
    \pgfsetcornersarced{\pgfpoint{5pt}{5pt}}
    \pgfpathlineto{\pgfpoint{\pgf@xb}{\pgf@ya}}
    \pgfpathclose
 }
}
\makeatother



\usepackage{xcolor}

\title[Sparse $4$-critical graphs have low circular chromatic number]{Sparse $4$-critical graphs have low circular chromatic number}

\author[Moore]{Benjamin Moore}
\address[Benjamin Moore]{Department of Combinatorics and Optimization, University of Waterloo, Waterloo, ON, Canada}  
\email{brmoore@uwaterloo.ca}

\date{}

\newtheorem{thm}[equation]{Theorem}
\newtheorem{lemma}[equation]{Lemma}
\newtheorem{prop}[equation]{Proposition}
\newtheorem{conj}[equation]{Conjecture}
\newtheorem{cor}[equation]{Corollary}
\newtheorem{claim}{Claim}
\newtheorem{question}[equation]{Question}

\newtheorem{definition}[equation]{Definition}
\newtheorem{obs}[equation]{Observation}

\theoremstyle{definition}

\newtheoremstyle{case}{}{}{\normalfont}{}{\itshape}{\normalfont:}{ }{}

\theoremstyle{case}

\numberwithin{equation}{section}

\date{}

\begin{document}

\begin{abstract}

Kostochka and Yancey proved that every $4$-critical graph $G$ has $e(G) \geq \frac{5v(G) - 2}{3}$, and that equality holds if and only if $G$ is $4$-Ore. We show that a question of Postle and Smith-Roberge implies that every $4$-critical graph with no $(7,2)$-circular-colouring has $e(G) \geq \frac{27v(G) -20}{15}$. Here, for integers $p$ and $q$ where $\frac{p}{q} \geq 2$, we say that $G$ admits a $(p,q)$-circular colouring if there is a map $f:V(G) \to \{0,\ldots,p-1\}$ such that for any edge $uv \in E(G)$,  $q \leq |f(u) -f(v)| \leq p-q$. We show that if this question is true, then it is best possible in the sense that for any integers $p$ and $q$, where $2 \leq \frac{p}{q} < \frac{7}{2}$, there exists a $4$-critical graph $G$ with no $(p,q)$-colouring that has $e(G) < \frac{27v(G)-20}{15}$. Towards the question, we prove that every $4$-critical graph $G$ with no $(7,2)$-colouring has $e(G) \geq \frac{17v(G)}{10}$ except for $K_{4}$ and the wheel on six vertices. A consequence of this is that the complement of every $4$-critical graph $G$ with $e(G) < \frac{17v(G)}{10}$ aside from $K_{4}$ and the wheel on six vertices has a Hamiltonian cycle. 
 
We prove additional structural statements about $4$-critical graphs with no $(7,2)$-colouring. Let $G$ be a $4$-critical graph with no $(7,2)$-colouring, and let $D_{3}(G)$ be the subgraph induced by the vertices of degree $3$ in $G$. We prove that every connected component of $D_{3}(G)$ is isomorphic to either a path, a claw, or an odd cycle. In the event $D_{3}(G)$ contains an odd cycle, we show that $G$ is isomorphic to an odd wheel. In fact, we show that for any  $k$-critical graph $G$ with $k \geq 4$, that if the subgraph induced by the vertices of degree $k-1$ contains a $K_{k-1}$, then either $G$ is isomorphic to $K_{k}$, or $G$ admits a $(2k-1,2)$-colouring. 

Lastly, we construct examples of $4$-critical graphs with no $(7,2)$-colouring where $D_{3}(G)$ has components isomorphic to either a claw or arbitrarily long paths. 

\end{abstract}

\maketitle

\section{Introduction}

The topic of this paper is $4$-critical graphs. A graph $G$ is \textit{$k$-critical} if the chromatic number of $G$ is $k$, but every proper subgraph has chromatic number $k-1$. 

A remarkable result of Alexandr Kostochka and Matthew Yancey
is that $k$-critical graphs have a ``large" number of edges. More precisely, they proved:

\begin{thm}[\cite{oresconjecture}]
\label{KostochkaYancey}
Let $k \geq	 4$ and let $G$ be a $k$-critical graph. Then 
\[e(G) \geq \lceil \frac{(k+1)(k-2)v(G) - k(k-3)}{2(k-1)} \rceil.\]
\end{thm}

Here, we are using the notation that $e(G) = |E(G)|$ and $v(G) = |V(G)|$. We will use this notation throughout the article. Later, they strengthened their theorem by characterizing when equality holds:

\begin{thm}[\cite{tightnessore}]
\label{Orebound}
Equality holds in Theorem \ref{KostochkaYancey} if and only if $G$ is $k$-Ore. 
\end{thm}

Here, a graph is \textit{$k$-Ore} if it can be obtained from repeated Ore compositions of $K_{k}$. An \textit{Ore composition} of two graphs $G_{1}$ and $G_{2}$ is the graph $G$ obtained by deleting an edge $xy \in E(G_{1})$, taking a vertex $z$ in $V(G_{2})$, splitting it into two vertices $z'$ and $z''$ of positive degree, and then identifying $z'$ with $x$ and $z''$ with $y$. It follows that there are infinitely many $k$-Ore graphs, and hence Theorem \ref{KostochkaYancey} is tight infinitely often. These two theorems are useful at proving colouring results on sparse graphs. A particularly nice application is an exceptionally short proof of Gr{\"o}tzsch’s Theorem - that triangle-free planar graphs are $3$-colourable \cite{Shortproof} (see \cite{moreapplications} for more nice applications). 

A limitation of Theorem \ref{KostochkaYancey} arises from the following situation. We have a class of graphs which are sparse but have more edges than the bound given in Theorem \ref{KostochkaYancey} for some value of $k$, but we would still like to $k$-colour these graphs. Then unfortunately without further arguments Theorem \ref{KostochkaYancey} is not particularly useful. One way to get around this would be to identify some structural properties of the class of graphs, and argue that all $k$-critical graphs containing that structural property are more dense than we can guarantee from Theorem \ref{KostochkaYancey}. 

There are numerous results of this flavour. For instance, $k$-Ore graphs contain large cliques. Luke Postle conjectured that if a $k$-critical graph has no large cliques with respect to $k$, then the $k$-critical graph is more dense than expected.

\begin{conj}[\cite{postletrianglefree5crit}]
For every $k \geq 4$, there exists $\varepsilon_{k} > 0$ such that if $G$ is $k$-critical and does not contain a $K_{k-2}$ subgraph, then 
\[e(G) \geq (\frac{k}{2}- \frac{1}{k-1}+ \varepsilon_{k})v(G) - \frac{k(k-3)}{2(k-1)}\].
\end{conj}

The conjecture does not have any content $k=4$. The $k =5$ case was proven by Postle \cite{postletrianglefree5crit}. The $k=6$ was proven by Wenbo Gao and Postle \cite{6critk4free}, and the $k \geq 33$ was proven by Victor Larsen in his thesis \cite{largekthesis}. In the case where $k=4$, the question becomes interesting if you replace having no $K_{2}$-subgraph with having girth $5$, and in this case progress has been made in the constant term by Chun-Hung Lui and Postle \cite{4criticalgirth5} (while not available online, Postle claims to have made an improvement in the density in this case as well). 
 Tom Kelly and Postle investigated the density of critical graphs without large cliques \cite{kelly2019density} and showed you can obtain improvements to the density in this situation.

This paper deals with a similar style of problem, but with a less straightforward structural condition than simply having no large clique (or large girth). We impose a no-homomorphism condition on our $4$-critical graphs, which informally says that not only can our $4$-critical graph not be $3$-coloured, but it can not be coloured even given three and a half colours. We now make this notion precise.

 Given two graphs $G$ and $H$, a \textit{graph homomorphism} is an adjacency preserving map from $G$ to $H$. That is, a map $f:V(G) \to V(H)$ such that for every edge $uv \in E(G)$, we have $f(u)f(v) \in E(H)$. We write $G \to H$ if $G$ admits a homomorphism to $H$. It is easily seen that a graph $G$ has a $k$-colouring if and only if $G$ has a homomorphism to $K_{k}$. Hence homomorphisms generalize colouring. 
 
  An interesting class of homomorphism targets that refines complete graphs are the \textit{circular cliques}. Let $p$ and $q$ be positive integers such that $\frac{p}{q} \geq 2$. Then we say the \textit{$(p,q)$-circular-clique}, denoted $G_{p,q}$, has vertices $\{0,1,2,\ldots,p-1\}$ and an edge $ij$ if $q \leq |i-j| \leq p-q$. We say $G$ admits an \textit{$(p,q)$-circular-colouring} if $G$ admits a homomorphism to $G_{p,q}$. As there will be no confusion, we will refer to $(p,q)$-circular-colourings as $(p,q)$-colourings. It is easy to see that $K_{k}$ is isomorphic to $G_{k,1}$, and that the odd cycle on $2k+1$ vertices, $C_{2k+1}$, is isomorphic to $G_{2k+1,k}$.
We refer the reader to the following survey of circular colouring by Xuding Zhu for an comprehensive overview of circular colouring \cite{XudingSurvey1}. For ease, we will always use the labelling of $G_{p,q}$ given above. This gives rise to a different labelling than the obvious standard labelling one would use when it comes to odd cycles, however it is easier to use this labelling for circular colouring. 

Now we can introduce homomorphism critical graphs. For a fixed graph $H$, we will say that a graph $G$ is \textit{$H$-critical} if $G$ does not admit a homomorphism to $H$, but all proper subgraphs do. In \cite{c5criticalgraphs}, Zden{\v e}k Dvo{\v r}{\' a}k and Postle investigated $C_{5}$-critical graphs, and proved
 \begin{thm}[\cite{c5criticalgraphs}]
If $G$ is $C_5$-critical and not $K_{3}$, then $e(G) \geq \frac{5v(G)-2}{4}$.
 \end{thm}

They conjectured the following bound:

\begin{conj}
\label{C5conjecture}
If $G$ is $C_{5}$-critical, then $e(G) \geq  \frac{14v(G) -9}{11}$.
\end{conj}

Further, they observed using a standard Hell-Ne{\v s}et{\v r}il indicator construction (see \cite{Hcol}) that if this conjecture were true, it would generalize the $k=6$ case of Theorem \ref{KostochkaYancey}. To see this, consider any $6$-critical graph $G$ and construct the graph $G \ast P_4$ by taking every edge $uv$ of $G$, deleting it, and replacing the edge with a path on $4$ vertices (identifying the endpoints of the path with $u$ and $v$). It is easy to see that $G \ast P_4$ is $C_5$-critical. Further if $G$ is $6$-Ore, then the graph $G \ast P_4$ has exactly $\frac{14v(G) -9}{11}$ edges. As there are $C_{5}$-critical graphs that do not arise from this construction (for instance, $K_{3}$, or the construction in \cite{MACGILLIVRAY} can be used to find many such examples), the conjecture is a strengthening of the bound in Theorem \ref{KostochkaYancey}.

Later, Postle and Smith-Roberge asked the natural generalization of Conjecture \ref{C5conjecture}.

\begin{question}[\cite{EvelyneMasters}]
\label{bigquestion}
Is it true that if $G$ is $C_{2t+1}$-critical, for $t \geq 2$, then
\[e(G) \geq \frac{t(2t+3)v(G) - (t+1)(2t-1)}{2t^{2}+2t-1}?\]
\end{question}

By a similar construction as above, if Question \ref{bigquestion} is true for some $t$, then this generalizes Theorem \ref{KostochkaYancey} in the $k = 2t+2$ case. Towards the question when $t=3$, they proved:

\begin{thm}[\cite{EvelyneMasters}]
\label{C7critical}
Let $G$ be a $C_{7}$-critical graph. If $G$ is not $C_{3}$ or $C_{5}$, then 
\[e(G) \geq \frac{17v(G)-2}{15}\]
\end{thm}

Interestingly, the $t=3$ case not only would generalize the $k=6$ case of Theorem \ref{KostochkaYancey}, it also implies bounds on $G_{7,2}$-critical graphs. In particular,

\begin{obs}
\label{indicatorobservation}
If the $t=3$ case of Question \ref{bigquestion} is true, then every $G_{7,2}$-critical graph $G$ satisfies
\[e(G) \geq \frac{27v(G)-20}{15}.\]
\end{obs}

We prove this later. As we observe in a moment, every $4$-critical graph that has no $(7,2)$-colouring is $G_{7,2}$-critical. Thus the $t=3$ case of Question \ref{bigquestion} would imply that sparse $4$-critical graphs have low circular chromatic number. Recall the useful and easy fact that if $\frac{p}{q} \leq \frac{p'}{q'}$, then $G_{p,q} \to G_{p',q'}$.

\begin{obs}
Let $p$ and $q$ be integers such that $3 \leq \frac{p}{q} < 4$. Any $4$-critical graph with no $(p,q)$-colouring is $G_{p,q}$-critical. 
\end{obs}

\begin{proof}
Let $G$ be such a graph. By the assumption $G$ has no $(p,q)$-colouring. By $4$-criticality, for any edge $e \in E(G)$, $G-e \to K_{3}$, and $K_{3} \to G_{p,q}$ as $3 \leq \frac{p}{q}$. As homomorphisms compose, $G-e \to G_{p,q}$, and hence $G-e$ has a $(p,q)$-circular colouring. Therefore $G$ is $G_{p,q}$-critical. 
\end{proof}

We will show that the bound in Observation \ref{indicatorobservation} is sharp with respect to the choice of $\frac{7}{2}$.  

\begin{obs}
\label{sharpness}
For any integers $p$ and $q$ satisfying $3 \leq \frac{p}{q} < \frac{7}{2}$, there is a $4$-critical graph with no $(p,q)$-colouring but satisfies
\[e(G) < \frac{27v(G) -20}{15}.\] 
\end{obs}

We use the notation $W_{n}$ to denote the wheel on $n$ vertices. This is the graph obtained by taking a cycle $C_{n-1}$ and adding a vertex $h$ adjacent to all vertices in $C_{n-1}$. We will say a wheel is \textit{odd} if the number of vertices in $W_{n}$ is even. The \textit{claw} is the unique tree on four vertices with a vertex of degree $3$.  For any graph $G$, we will let $D_{k}(G)$ denote the subgraph induced by the vertices of degree $k$.  Now we can state the main result of the paper.
\begin{thm}
\label{maintheorem}
Let $G$ be a $4$-critical graph that does not have a $(7,2)$-colouring and is not isomorphic to $K_{4}$ or $W_{6}$.
Then 
\[e(G) \geq \frac{17v(G)}{10}.\]
\end{thm}

It follows immediately that the complement of all $4$-critical graphs with few edges have a Hamiltonian cycle. Recall, the complement of a graph $G$ is the graph $\bar{G}$ on the same vertex set, where $uv \in E(\bar{G})$ if and only if $uv \not \in E(G)$. We give a proof of the following observation that copies the same idea as a more general statement in \cite{complements}. 

\begin{obs}[\cite{complements}]
\label{hamcycle}
If a $4$-critical graph $G$ admits a $(7,2)$-colouring, then $\bar{G}$ contains a Hamiltonian cycle. 
\end{obs}

\begin{proof}
Let $f$ be a $(7,2)$-colouring of $G$. Observe that $f$ uses all seven colours, as otherwise we can recolour $f$ to a $3$-colouring of $G$, contradicting the fact that $G$ is $4$-critical (see \cite{HomBook} for justification if you cannot convince yourself of this fact). Let $C_{0},C_{1},\ldots,C_{6}$ be the colour classes of $f$. By the above observation no colour class is empty. Then for any vertex $v \in V(C_{i})$, there is no edge from $v$ to any vertex $u \in V(C_{i+1})$ for all $i \in \{0,\ldots,6\}$ (indices taken modulo $7$). Then in the complement, there are edges between all vertices in each colour class, and all edges between colour classes $C_{i}$ and $C_{i+1}$. It follows one can find a Hamiltonian cycle. 
\end{proof}

\begin{cor}
If $G$ is $4$-critical, $G$ is not isomorphic to $K_{4}$ or $W_{6}$, and $e(G) < \frac{17v(G)}{10}$, then $\bar{G}$ contains a Hamiltonian cycle. 
\end{cor}

\begin{proof}
By Theorem \ref{maintheorem}, $G$ has a $(7,2)$-colouring. By Observation \ref{hamcycle}, $\bar{G}$ has a Hamiltonian cycle. 
\end{proof}

Aside from Theorem \ref{maintheorem}, we prove some structural results about the types of graphs which can appear in $D_{3}(G)$. 

\begin{thm}
\label{GallaiTreestructure}
Let $G$ be a $4$-critical graph that does not have a $(7,2)$-colouring. Then either $G$ is isomorphic to an odd wheel, or every component of $D_{3}(G)$ is isomorphic to a path or a claw. Further there are $4$-critical graphs with no $(7,2)$-colouring that have components of $D_{3}(G)$ isomorphic to either a claw, or arbitrarily long paths. 
\end{thm}

We can generalize the ideas of part of Theorem \ref{GallaiTreestructure} to $k$-critical graphs. 

\begin{thm}
\label{kcriticality}
Let $G$ be a $k$-critical graph that does not have a $(2k-1,2)$-colouring. Then $D_{k-1}(G)$ does not contain a clique of size $k-1$.  
\end{thm}

In light of the above, we make the following strong conjecture.

\begin{conj}
Let $p$ and $q$ be integers where $3 < \frac{p}{q} < 4$. Let $G$ be a $4$-critical graph with no $(p,q)$-colouring. Then there exists positive rational numbers $\varepsilon_{p,q}$ and $c_{p,q}$ depending on $p$ and $q$ such that
\[e(G) \geq \frac{(5 + \varepsilon_{p,q}) -c_{p,q}}{3}.\]
\end{conj}

We give a brief overview of the proof of Theorem \ref{maintheorem} and Theorem \ref{GallaiTreestructure}. From a fundamental result of Gallai, we know that the subgraph $D_{3}(G)$ has every block isomorphic to either an odd cycle or a clique. It is easy to see that the cliques have size at most $3$, or $G$ is isomorphic to $K_{4}$. The first part of the proof then is to show that if any block is isomorphic to an odd cycle, then $G$ is isomorphic to an odd wheel. This is done by taking an odd wheel $C$, deleting it, and characterizing when $3$-colouring of $G-C$ extends to a $(7,2)$-colouring. In the cases where we cannot extend, the neighbours of $C$ will form an independent set, and if $G$ is not a wheel, we will be able to reconfigure the colouring so that it will be able to extend to a $(7,2)$-colouring. 

The next part of the proof is to show that assuming we have no odd cycle blocks or $K_{4}$ blocks, that every component of $D_{3}(G)$ is isomorphic to a path or has at most $4$ vertices. This proof follows the same themes as the odd cycle reduction. We delete vertices from $D_{3}(G)$ and ask when we can extend a $3$-colouring to a $(7,2)$-colouring, and show that unless each component is isomorphic to a path or a claw, we can always extend.  

Once we have done this, we show that a vertex minimal counterexample to Theorem \ref{maintheorem} cannot contain a claw component.
Then we use reconfiguration arguments to show that vertices close to path components in $D_{3}(G)$ have reasonably large degree, and finish the proof via discharging.

The structure of the paper is as follows. In Section \ref{prelims} we introduce the basics of circular colouring that will be needed for the paper. We also prove Observation \ref{sharpness} and give the examples of $4$-critical graphs with no $(7,2)$-colouring whose components of the Gallai Tree are isomorphic to claws and arbitrarily long paths, proving part of Theorem \ref{GallaiTreestructure}. 
In Section \ref{indicators} we review the Hell-Ne{\v s}et{\v r}il indicator construction and prove Observation \ref{indicatorobservation}. In Section \ref{Oddcycle}, we prove that for a $4$-critical graph $G$ with no $(7,2)$-colouring, if $D_{3}(G)$ contains an odd cycle, then $G$ is isomorphic to an odd wheel. In Section \ref{detourtokcrit} we prove that for given a $k$-critical graph $G$, $k \geq 4$, if $D_{k-1}(G)$ contains a clique of size $k-1$, then either $G$ is isomorphic to $K_{k}$, or admits a $(2k-1,2)$-colouring. In Section \ref{reductiontopathssection}, we prove that the Gallai Tree of $4$-critical graphs with no $(7,2)$-colouring can only have components isomorphic to claws, paths, or $G$ is isomorphic to an odd wheel. In Section \ref{remainingarguments} we prove that path components are close to vertices of large degree. In Section \ref{finishedhard} we provide the discharging argument to finish the proof.

\section{Preliminaries, Sharpness and Examples}
\label{prelims}
In this section we collect the basic results from colouring that we will make use of throughout the paper. We also exhibit a $4$-critical graph $G$ with a $(7,2)$-colouring which has $e(G) < \frac{27v(G) -20}{15}$, showing that the $t=3$ case of Question \ref{bigquestion} is sharp even when restricted to $4$-critical graphs (after applying the indicator construction). We also give some new examples of $4$-critical graphs with no $(7,2)$-colouring.

\subsection{A sparse $4$-critical graph with a $(7,2)$-colouring}

We will need to know what happens when we have a circular colouring that does not use all all of the colours. For this we need the notion of lower parents.

\begin{definition}
Let $p$ and $q$ be positive integers where $\frac{p}{q} \geq 2$, and $\gcd(p,q) =1$. The unique positive integers $p'$ and $q'$ where $p' < p$ that satisfy the equation $pq' - q'p = 1$ are called the \textit{lower parents of $p$ and $q$}. 
\end{definition}

We will say that two graphs $G$ and $H$ are \textit{homomorphically equivalent} if $G \to H$ and $H \to G$.

\begin{lemma}[\cite{HomBook} Lemma $6.6$]
Let $p$ and $q$ be positive integers that are relatively prime, and satisfy $\frac{p}{q} \geq 2$. Let $p'$ and $q'$ be the lower parents of $p$ and $q$. Then for any vertex $x \in V(G_{p,q})$, the graph $G_{p,q}-x$ is homomorphically equivalent to $G_{p',q'}$. 
\end{lemma}

Hence we have the following corollary:

\begin{cor}
\label{lowerparents}
Let $p$ and $q$ be relatively prime. If a graph $G$ admits a $(p,q)$-colouring that does not use all $p$ colours, then $G$ admits a $(p',q')$-colouring where $p',q'$ are the lower parents of $p$ and $q$.  
\end{cor}

Observe that the lower parents of $7$ and $2$ are $3$ and $1$, and hence every $(7,2)$-colouring which is not surjective is a $3$-colouring. With this we will prove that there is a graph on seven vertices that has eleven edges and circular chromatic number $\frac{7}{2}$. 

Let the \textit{Moser Spindle}, denoted $M$, be the unique $4$-Ore graph on $7$ vertices. For clarity, we have $V(M) = \{a,b,c,d,e,f,g\}$, and $E(M) = \{ab,ac,af,ag,bc,bd,cd,de,ef,eg,fg\}$.

\begin{obs}
The Moser Spindle has seven vertices, eleven edges, and circular chromatic number $\frac{7}{2}$. 
\end{obs}

\begin{proof}
As the Moser Spindle is $4$-Ore, it is $4$-critical, and hence does not have a $3$-colouring. The map where we colour $a$ with $0$, $f$ with $2$, $g$ with $4$, $e$ with $6$, $d$ with $1$, $b$ with $5$, $c$ with $3$ is a $(7,2)$-colouring. Finally it is easy to check that for every $p \in \{4,5,6\}$ there is no integral $q$ where $\gcd(p,q)=1$ such that $3 \leq \frac{p}{q} < \frac{7}{2}$. Therefore $\chi_{c}(M) = \frac{7}{2}$. 
\end{proof}

An astute reader may realize that the Moser Spindle is not isomorphic to $G_{7,2}$, and this implies that there are graphs such are strict subgraphs of $G_{7,2}$ with circular chromatic number $\frac{7}{2}$. This turns out to be the case for any tuples $(p,q)$ unless $p = 2k+1$ and $q =k$, or $q=1$. Rather surprisingly, you can find a subgraph with roughly $O(\sqrt{e(G_{p,q})})$ edges on $p$ vertices with circular chromatic number $\frac{p}{q}$ \cite{zhu_1999}. By appealing to the Kostochka-Yancey Theorem, the Moser Spindle has the fewest edges for a graph on $7$ vertices that is also $4$-critical and has circular chromatic number $\frac{7}{2}$. Now the sharpness claim follows immediately. 

\begin{cor}
For integers $p$ and $q$, satisfying $2 \leq \frac{p}{q} < \frac{7}{2}$, there exists a graph that is $4$-critical with no $(p,q)$-colouring that has
\[e(G) < \frac{27v(G) -20}{15}.\]
\end{cor}

\begin{proof}
The Moser Spindle has circular chromatic number $\frac{7}{2}$, seven vertices and eleven edges. Observe that $11 < \frac{169}{15}$. 
\end{proof}

Of course, this is not the most satisfying sharpness example. It would be much more interesting if an infinite family were found.  Nevertheless, it does show that the bound in Observation \ref{indicatorobservation} is sharp with respect to the values of $p$ and $q$. 

\subsection{Useful Background Lemmas}
We will need one more idea from circular colouring. As in circular colouring colours may be distinct but still not allowed to be adjacent, it is convenient to be able to talk about intervals of colours. If we have $p$ colours, and integers $i,j \in \{0,1,\ldots,k-1\}$ we denote $[i,j]$ as the set of colours $\{i,i+1,\ldots,j\}$ where the values are reduced modulo $p$. When $p$ is fixed, we will assume intervals are taken modulo $p$. 

Given a graph $G$ and a vertex $v$, let $N_{G}(v)$ denote the neighbourhood of $v$ in $G$. If there is no possibility of confusion we will just use $N(v)$. Observe that in any $(p,q)$-colouring $f$, for any vertex $v$ we have
\[f(v) \in \bigcap_{u \in N_{G}(v)}N_{G_{p,q}}(f(u))\]

Given a graph $G$, and an induced subgraph $F$ of $G$ equipped with a $(p,q)$-colouring $f$ of $F$, we say that the set of available colours for $v$ in $G$ is $[0,p-1]$ if $v$ has no neighbours in $F$, and
\[\bigcap_{u \in N_{F}(v)}N_{G_{p,q}}(f(u))\]
otherwise. 

A very useful fact is that when $2 < \frac{p}{q} < 4$, the set of available colours is always an interval.

\begin{lemma}[\cite{rickjon}]
If $\frac{p}{q} < 4$, then for any graph $G$, any $(p,q)$-colouring of $G$, and any vertex $v \in V(G)$, the set of available colours of $v$ is an interval. 
\end{lemma}

We will use this fact without reference. 
We also record some facts about $k$-critical graphs which we use without reference. We start off with a very easy observation. 

\begin{obs}
A $k$-critical graph is $(k-1)$-edge-connected. In particular, the minimum degree of a $k$-critical graph is at least $k-1$. 
\end{obs}

Recall that a \textit{block} of a graph is a maximal $2$-connected subgraph. The Gallai-Tree Theorem gives structure to subgraph induced by the vertices of degree $k-1$.

\begin{thm}[\cite{GallaiForests}, Gallai-Tree Theorem]
\label{GallaiTreeTheorem}
Let $G$ be a $k$-critical graph. Then every block of $D_{k-1}(G)$ is a clique or an odd cycle. 
\end{thm}

We will call the graph $D_{k-1}(G)$  the \textit{Gallai Tree of $G$}. We will use Theorem \ref{GallaiTreeTheorem} without reference. For $4$-critical graphs, this implies the following.

\begin{cor}
In a $4$-critical graph that is not $K_{4}$, every block of the Gallai Tree is either isomorphic to $K_{1}$, isomorphic to $K_{2}$, or an odd cycle.
\end{cor}

\begin{proof}
Clearly $K_{4}$ is $4$-critical, and hence the largest clique a $4$-critical graph can have is $K_{4}$. The rest follows immediately from the Gallai-Tree Theorem.
\end{proof}

Lastly, we introduce some notation which is mostly standard. For a set of vertices $X$, we let $N(X)$ denote the \textit{neighbourhood of $X$}, which is the set of vertices adjacent to a vertex in $X$, but not in $X$. We let the \textit{degree} of a vertex be $\deg(v) = |N(v)|$. We will use the notation $\deg_{t}(v)$ to denote the number of vertices in the neighbourhood of $v$ with degree $t$. If $G$ is equipped with a $k$-colouring, we let $N_{t}(X)$ denote the set of neighbours of $X$ coloured $t$. For ease throughout the paper, when given a $3$-colouring, we will always assume the colours used are from the set $\{0,2,4\}$. This is so we can extend to a $(7,2)$-colouring without any cumbersome change in values. 

\subsection{Examples of $4$-critical graphs with no $(7,2)$-colourings}

In this section we collect the known examples from the literature of $4$-critical graphs with no $(7,2)$-colouring, and provide an operation which preserves the property of being $4$-critical and having no $(7,2)$-colouring. This operation produces to the best of my knowledge a new infinite family of $4$-critical graphs with no $(7,2)$-colouring (the family is surely known before, however the new property is that they have no $(7,2)$-colouring). In particular this family demonstrates that there are $4$-critical graphs with no $(7,2)$-colourings with claw components or arbitrarily long paths in their Gallai Tree.

Before we describe the operation, we collect some examples from the literature. 
Recall that the \textit{circular chromatic number} is defined to be
\[\chi_{c}(G) := \inf \{\frac{p}{q} \, | \, G \to G_{p,q}\}.\]

It follows from a theorem in \cite{zhusufficient} that if the complement of a graph $G$ is disconnected, then $\chi_{c}(G) = \chi(G)$. As the complement of an odd wheel is disconnected, we obtain our first example. 

\begin{obs}
All odd wheels have circular chromatic number four. 
\end{obs}

Our next example uses a well known construction. Given a graph $G$, we let $M(G)$ denote the \textit{Mycielski} of $G$, where $V(M(G)) = V(G) \cup V'(G) \cup \{u\}$, $V'(G) = \{x' | x \in V(G)\}$, and $E(M(G)) = E(G) \cup \{xy' \,|\, xy \in E(G)\} \cup \{y'u \,|\, y' \in V'(G)\}$. 

\begin{thm}[\cite{Mycielskigraphs}]
 For all integers $k$, $\chi_{c}(M(C_{2k+1})) = 4$, and $M(C_{2k+1})$ is $4$-critical. 
\end{thm}

Observe that the Gallai Tree of the Mycielski of an odd cycle is a collection of isolated vertices. We do not define the family here as they do not contain vertices of degree $3$, but for completeness we note that in \cite{large4crit}, an infinite family of $4$-regular $4$-critical graphs with $\chi_{c}(G) = \chi(G)$ was found. 

Now we describe an operation which preserves $4$-criticality and not having a $(7,2)$-colouring. The operation given here generalizes the operation called the ``iterated Mycielski" in  \cite{iteratedMycielski} when restricted to $4$-critical graphs.  

\begin{definition}
\label{c6expansion}
Let $G$ be a graph. Let $v \in V(G)$ such that $N(v) = \{x,y,z\}$. The \textit{$C_{6}$-expansion of $G$ with respect to $v$} is the graph $G'$ obtained by deleting $v$ from $G$, and adding four new vertices $x',y',z',w$ with edges $x'y,x'z,x'w, y'x, y'z, y'w, z'x,z'y$ and $z'w$. 
\end{definition}

\begin{lemma}
\label{clawexpansion}
Let $G$ be the $C_{6}$-expansion of a graph $H$ at a vertex $v$, where $H$ is $4$-critical and has no $(7,2)$-colouring. Then $G$ is $4$-critical and has no $(7,2)$-colouring. 
\end{lemma}

\begin{proof}
Let the neighbours of $v$ in $H$ be $x,y,z$, with the new vertices in $G$ being $x',y',z',w$ with adjacencies as in Definition \ref{c6expansion}. 
First we observe that $\chi(G) \leq 4$. Let $f$ be a $4$-colouring of $H$. Let $f'$ be a colouring of $G$ where for all $t \in V(G) \setminus \{x',y',z',w\}$ let $f'(t) = f(t)$, for $\{x',y',z'\}$ let  $f'(x') = f'(y')='f(z') = f(v)$ and finally give $w$ any colour that is not $f(v)$. This is a $4$-colouring of $G$, and hence $\chi(G) \leq 4$.

Now we prove that $\chi_{c}(G) > \frac{7}{2}$. Suppose not and let $f$ be a $(7,2)$-colouring of $G$. Observe that if $f(x') = f(y') = f(z')$, then  by identifying $x'$, $y'$ and $z'$ into one vertex and deleting $w$ we obtain a $(7,2)$-colouring of $H$, a contradiction. 

 Suppose that $f(x') = 0$. Further suppose that $f(y') = 0$. Then the image of the neighbourhood of $z'$ is contained in the neighbourhood of $0$ in $G_{7,2}$, and so we can recolour $z'$ to $0$. But then there exists a $(7,2)$-colouring of $G$ where $x',y'$ and $z'$ receive the same colour, a contradiction. Thus we can assume that $f(y') \neq 0$, and by symmetry $f(z') \neq 0$. If $f(y') = 1$ and $f(z') = 6$, then again the image of the neighbourhood of both $y'$ and $z'$ is contained in the neighbourhood of $0$ in $G_{7,2}$, and so we can recolour $y'$ and $z'$ to $0$, and again obtain a $(7,2)$-colouring of $H$, a contradiction. Now assume that $f(y') =2$. Then $f(z') \in \{0,1,2,3\}$ or else there is no available colour for $w$. By the previous cases, it follows that $f(z') = 3$. But this implies that we can recolour $z'$ and $x'$ to $2$, and again contradict that $H$ has no $(7,2)$-colouring. All other cases follow similarly, and thus $G$ has no $(7,2)$-colouring. Observe this also implies that $\chi(G) \geq 4$, and hence $\chi(G) = 4$. 

Therefore to finish the proof, we just need to show that for every edge $e \in E(G)$, $G-e$ is $3$-colourable. First let $e$ be incident to $w$. Let $f$ be a $3$-colouring of $H-v$. We extend $f$ to a $3$-colouring of $G$ by letting $f(x) = f(x')$, $f(y) = f(y')$, $f(z) = f(z')$. As  we deleted an edge incident to $w$, there is an available colour left for $w$ and thus we have a $3$-colouring of $G-e$. 

Now suppose that $e$ is incident to $x'$ but not $w$. Without loss of generality, let $e = zx'$. Then take a $3$-colouring $f$ of $H-v$, and extend $f$ by letting $f(x') = f(z') = f(z)$ and $f(y') = f(y)$, and giving $w$ a colour that is left over. By symmetry we can assume that $e$ is not incident to any of $x',y',z'$. 

Now let $f$ be a $3$-colouring of $H-e$. Then extend $f$ to a $3$-colouring of $G$ by colouring $x',y',z'$ the same colour as $v$, and then giving $w$ any colour left over. Thus it follows that $G-e$ has a $3$-colouring for every edge $e$, and thus $G$ is $4$-critical with no $(7,2)$-colouring. 
\end{proof}

\begin{cor}
There is a $4$-critical graph with no $(7,2)$-colouring whose Gallai Tree is isomorphic to a claw. 
\end{cor}

\begin{proof}
Let $H = K_{4}$, and $G$ a $C_{6}$-expansion of any vertex in $H$. Then by Lemma \ref{clawexpansion}, $G$ is $4$-critical, has no $(7,2)$-colouring, and it is easily seen that the Gallai Tree of $G$ is just the claw. 
\end{proof}

\begin{cor}
For any odd positive integer $t$, there is a $4$-critical graph with no $(7,2)$-colouring whose Gallai Tree contains a component isomorphic to a path of length $t$. 
\end{cor}

\begin{proof}
Fix an odd positive integer $t$. Let $G$ be a $C_{6}$-expansion of any of the degree $3$ vertices in $W_{t+5}$. Then the Gallai Tree of $G$ contains a component isomorphic to the claw, and a path of length $t$. 
\end{proof}

To the best of my knowledge, the graphs that can be obtained via $C_{6}$-expansions starting from odd wheels and Mycielski construction and the family in \cite{large4crit} are the only known $4$-critical graphs with no $(7,2)$-colouring. More examples would be helpful in trying to understand the structure of this class of graphs.

\section{Hell-Ne{\v s}et{\v r}il indicator constructions}
\label{indicators}
In this section we prove that the $t=3$ case of Question \ref{bigquestion} implies that $4$-critical graphs with no $(7,2)$-colouring satisfy $e(G) \geq \frac{27v(G) -20}{15}$. 

We first review the basics of the indicator construction. Let $I$ be a graph with distinguished vertices $x$ and $y$, and suppose that there is an automorphism which sends $x$ to $y$. Let $G$ be a graph. We will say that $G \ast I$ is the graph obtained by taking every edge $uv \in E(G)$, deleting the edge, and adding the graph $I$ to $G$ where we identify $u$ with $x$ and $v$ with $y$. Observe this is well defined because there is an automorphism of $I$ which sends $x$ to $y$.

\begin{definition}
Let $G$, $H$, and $I$ be graphs, where $V(H) = V(G)$, and $x$ and $y$ are distinguished vertices of $I$, where there is an automorphism of $I$ which sends $x$ to $y$. Suppose that for any edge $uv \in E(G)$, there exists a homomorphism $f: I \to H$ such that $f(x) = u$ and $f(y) = v$. Then we say that $(I,x,y)$ is an \textit{indicator} for $G$ and $H$.
\end{definition}

The following is easily verified from the definition. 

\begin{lemma}[\cite{HomBook}, Lemma $5.5$]
\label{indicatorconstruction}
Suppose $(I,x,y)$ is an indicator for $G$ and $H$. Then for any graph $K$, $K \to G$ if and only if $K \ast I \to H$. 
\end{lemma}

We can use Lemma \ref{indicatorconstruction} to deduce the non-existence of homomorphisms in some instances. 

\begin{cor}
\label{nohomlemma}
Suppose that $(I,x,y)$ is an indicator for $G$ and $H$. If $K$ is $G$-critical, then $K \ast I  \not \to H$. 
\end{cor}

\begin{proof}
$K$ is $G$-critical, so $K \not \to G$. By Lemma \ref{indicatorconstruction}, this implies that $K \ast I \not \to H$. 
\end{proof}

This is of course not useful unless there exists indicator constructions. Here is a particularly useful class of indicators.

\begin{cor}[\cite{HomBook}, proof of Corollary $5.6$]
Let $I$ be the path of length $k-2$ with endpoints $x,y$. Then $(I,x,y)$ is an indicator for $K_{k}$ and $C_{2k+1}$. 
\end{cor}

This so far is not useful for critical graphs. However, one can observe that path indicators with the endpoints as the distinguished vertices preserve criticality. 

\begin{prop}
Let $(I,x,y)$ be an indicator for $G$ and $H$, where $I$ is a  path with at least one edge and $x$ and $y$ are the two endpoints for the path. Let $K$ be a $G$-critical graph. Then $K \ast I$ is $H$-critical. 
\end{prop}

\begin{proof}
From Corollary \ref{nohomlemma}, we have that $K \ast I \not \to H$. Now consider any edge $e \in K \ast I$. Then $e$ is contained in a copy of $I$, where this copy of $I$ replaced an edge $e'$ in $K$. By $G$-criticality, $K-e' \to G$. Hence $(K-e') \ast I \to H$ by Lemma \ref{indicatorconstruction}. However, as $I$ is a path, $K \ast I - e$ has vertices of degree $1$ (or is $(K-e') \ast I$, in which case we are done). Now we can map $K \ast I -e \to (K-e') \ast I$ by repeatedly mapping the degree one vertices in the copy of $I$ containing $e$ onto some vertex adjacent to their neighbour. But then as homomorphisms compose, $K \ast I -e \to H$, and hence $K \ast I$ is $H$-critical. 
\end{proof}

As notation let $P_{n}$ denote the path on $n$ vertices. Now to finish the intended goal of the section, we prove that that $P_{4}$ with the endpoints as distinguished vertices is an indicator for $G_{7,2}$ and $C_{7}$. 

\begin{obs}
Let $P_{4}$ be a path with endpoints $x$ and $y$. Then $(P_{4},x,y)$ is an indicator for $G_{7,2}$ and $C_{7}$.
\end{obs}

\begin{proof}
We just need to check the possible homomorphisms of $P_{4}$. Suppose $V(P_{4}) = \{x,x_{1},x_{2},y\}$, with edges $xx_{1},x_{1}x_{2},x_{2}y$. As $G_{7,2}$ is vertex transitive, it suffices to consider the case when we colour $x$ with $0$. The following are $C_{7}$-colourings of $P_{4}$ which give the necessary adjacencies. The colouring where we colour $x$ with $0$, $x_{1}$ with $3$, $x_{2}$ with $6$ and $y$ with $2$. The colouring where we colour $x$ with $0$, $x_{1}$ with $3$, $x_{2}$ with $0$ and $y$ with $3$. The colouring where we colour $x$ with $0$, $x_{1}$ with $4$, $x_{2}$ with $0$, and $y$ with $4$. The colouring  where we colour $x$ with $0$, $x_{1}$ with $4$, $x_{2}$ with $1$ and $y$ with $5$. 

Now we just need to show the non-adjacencies. Suppose that both $x$ and $y$ are coloured $0$. Then both of $x_{1}$ and $x_{2}$ would need to get a colour from $\{3,4\}$, but that is impossible. 

Suppose that $y$ is coloured $1$, then $x_{1}$ must be coloured $4$, as if it is coloured $3$ we cannot colour $x_{2}$ in a way that will be compatible with $y$ being coloured $1$. But if $x_{1}$ is coloured $4$, none of the neighbours of $4$ in $C_{7}$ are adjacent to $1$, and $y$ cannot be coloured $1$. The analysis is the same if $y$ is coloured $6$.
\end{proof}

Therefore we have the following corollary

\begin{cor}
\label{indicatorcriticality}
For any $G_{7,2}$-critical graph $G$, the graph $G \ast P_{4}$ is $C_{7}$-critical. 
\end{cor}

Now we can prove the observation.

\begin{obs}
If the $t=3$ case is true in Question \ref{bigquestion}, then for every $G_{7,2}$-critical graph $G$,
\[e(G) \geq \frac{27v(G) - 20}{15}.\]
\end{obs}

\begin{proof}
Let $G$ be a $G_{7,2}$-critical graph. By Corollary \ref{indicatorcriticality}, the graph $G \ast P_{4}$ is $C_{7}$-critical. Observe that $e(G \ast P_{4}) = 3e(G)$ and $v(G \ast P_{4}) = 2e(G) + v(G)$. Appealing the $t=3$ case of Question \ref{bigquestion}, we have
\[e(G \ast P_{4}) \geq \frac{27v(G \ast P_{4}) -20}{23}.\]
Thus
\[3e(G) \geq \frac{27(2e(G) + v(G)) -20}{23}.\]
Rearranging we have
\[69e(G) \geq 54e(G) +27v(G) -20.\]
Now simplifying we have
\[e(G) \geq \frac{27v(G) -20}{15}\]
as desired. 
\end{proof}

If we apply the same analysis using the bound on the density of $C_{7}$-critical graphs in Theorem \ref{C7critical}, we get a bound on $G_{7,2}$-critical graphs which to the best of my knowledge is the best known (however, this bound does not even beat the Kostochka-Yancey bound for $4$-critical graphs - which suggests many improvements should be possible). 

\begin{cor}[\cite{EvelyneMasters}]
If $G$ is a $G_{7,2}$-critical graph, then 
\[e(G) \geq \frac{17v(G)-2}{11}.\]
\end{cor}

\section{Odd cycles in the Gallai-Tree}
\label{Oddcycle}
The point of this section is to prove that the class of $4$-critical graphs with no $(7,2)$-colouring and whose Gallai Tree contains an odd cycle is exactly the class of odd wheels. 

The set up is to first prove a series of list colouring claims, which will allow us to assert that the neighbours of an odd cycle in the Gallai Tree form an independent set. If the independent set has size $1$, then the graph is an odd wheel, and otherwise using a reconfiguration argument we will be able to find a $(7,2)$-colouring. We remark that this set up is similar to the notion of collapsible and cocollapsible sets given in \citep{Oredegree7}, however the additional reconfiguration argument allows us to assert that we only obtain odd wheels (they obtain this for a vertex minimal counterexample).

We start off with some definitions. A \textit{$k$-list-assignment} $L$ is a function which assigns a set of at least $k$ colours to each vertex (however, without loss of generality we will always assume that each list is size exactly $k$). For a vertex $v$, we will denote $L(v)$ as the \textit{list of $v$}. A \textit{$4$-interval-list assignment} $L$ is a $4$-list-assignment where for all $v \in V(G)$, $L(v) \subseteq \{0,1,2,3,4,5,6\}$, and each list has size at least four, and contains an interval of size at least $4$. A list assignment is \textit{uniform} if all vertices receive the same list. An \textit{$L$-colouring} is a proper colouring where each vertex $v$ gets a colour from $L(v)$. An \textit{$L$-$(7,2)$-colouring} is an $L$-colouring which is also a $(7,2)$-colouring. Given a graph $G$ equipped with a list assignment $L$, and a subgraph $H$ of $G$, the list assignment induced by $H$ is simply the list assignment $L$ on the vertices of $H$. A vertex is \textit{precoloured} if $|L(v)| = 1$.

\subsection{$4$-interval-list-colouring paths}

The main point of this subsection is to characterize when we can list colour paths under the assumption that the endpoints have constrained lists. In particular we will characterize when we can colour $P_{n}$ when both of the endpoints have a list of size $2$ that forms and interval, and the internal vertices have lists of size $4$ that form intervals. 

We start with an easy observation. 

\begin{prop}
\label{pathclaim}
Let $P_{n}$ be a path with endpoints $x$ and $y$ (possibly not distinct if $n=1)$. Let $L$ be a list assignment where $x$ is precoloured from $\{0,\ldots,6\}$, and the list assignment induced on $P_{n} -x$ is a $4$-interval list assignment. Then $P_{n}$ is $L$-$(7,2)$-colourable. 
\end{prop}

\begin{proof}
We proceed by induction on $n$. If $n=1$, then the precolouring is an $L$-$(7,2)$-colouring. So $n \geq 2$.
Let $x'$ be the neighbour of $x$. By the pigeon hole principle, $|L(x') \cap N_{G_{7,2}}(L(x))| \geq 1$. Now colour $x'$ with some colour from $L(x') \cap N_{G_{7,2}}(L(x))$ and delete $x$. If $n=2$ then we are done, and otherwise the result follows by induction. 
\end{proof}

Note that it is possible to satisfy the hypothesis of the above claim and have exactly one $L$-$(7,2)$-colouring.
Now instead of precolouring one end of the path, we will restrict both endpoints of the path but not as severely.

\begin{lemma}
\label{sizetwoclaimsimplified}
Let $P_{n}$ be a path with endpoints $x$ and $y$. Suppose $L$ is a list assignment such that $P_{n} - x -y$ induces a $4$-interval-list assignment, $L(x)$ and $L(y)$ both forming intervals modulo $7$,  $|L(x)| \geq 2$ and $|L(y)| \geq 3$. Then there exists an $L$-$(7,2)$-colouring. 
\end{lemma}

\begin{proof}
We proceed by induction on $n$. If $n=1$, the claim is trivial. If $n = 2$, then let $c \in L(x)$. Unless $L(y) = \{c-1,c,c+1\}$, then we can colour $x$ with $c$ and extend the colouring. If $L(y) = \{c-1,c,c+1\}$, colour $x$ with a colour in $L(x) -c$, and extend the colouring in any fashion. Now we can assume that $n \geq 3$. Let $u$ be the neighbour of $x$. Observe that if there is a $c \in L(x)$ such that not all of $\{c-1,c,c+1\}$ are contained in $L(u)$, we can colour $x$ with $c$, delete $x$ and apply induction. Without loss of generality suppose that $L(x)= \{0,1\}$. Then by the above observation, unless $L(u) = \{6,0,1,2\}$, we can colour $x$ with either $0$ or $1$ and apply induction. Thus we may assume $L(u) = \{6,0,1,2\}$. Now let $v$ be the neighbour of $u$ which is not $x$. If $L(v)$ does not contain all of $\{1,2,3\}$, then we can colour $x$ with $0$, $u$ with $2$, and the set of available colours for $v$ has size at least $2$, so we can apply induction (or simply colour $v$ and finish the colouring if $n = 3$). Thus $L(v)$ contains all of $\{1,2,3\}$. Therefore $L(v)$ is one of three possible lists, $\{0,1,2,3\}$, $\{1,2,3,4\}$, or $\{1,2,3\}$. Regardless of which list $L(v)$ is, colour $x$ with $1$ and $u$ with $6$. In all cases, either we can finish the colouring of the path, or the set of available colours for $v$ is at least $2$, and we can apply induction.  
\end{proof}

The above lemma is best possible in the sense that we cannot make both endpoints have list size $2$, even if they form an interval. To see this, consider the following assignment of $P_{3}$ with vertices $x,y,z$ where we have edges $xy$ and $yz$. Let $L(x) = \{0,1\}$, $L(y) = \{5,6,0,1\}, L(z) = \{5,6\}$. It is easy to see there is no $L$-$(7,2)$-colouring.

\subsection{$4$-interval list colouring cycles}

Now we will turn our focus onto proving list colouring claims of odd cycles (or in some cases cycles). We now give a definition which is cooked up just to be able to apply Proposition \ref{pathclaim}. 

\begin{definition}
A $4$-interval-list-assignment $L$ of $C_{n}$ is \textit{safe} if for some edge $uv \in E(C_{n})$, there is a colour $c \in L(u)$ such that none of $\{c-1,c,c+1\}$ reduced modulo $7$ are in $L(v)$.
\end{definition}

\begin{obs}
\label{safeclaim}
Every safe $4$-interval-list-assignment of $C_{n}$ admits an $L$-$(7,2)$-colouring. 
\end{obs}

\begin{proof}
Pick an edge $uv \in E(C_{n})$ such that there is a colour $c \in L(u)$ where none of $\{c-1,c,c+1\}$ reduced modulo $7$ are in $L(v)$. Now consider $C_{n}-uv$. Colour $u$ with $c$. Then we satisfy the conditions of Proposition \ref{pathclaim}, so consider any $L$-$(7,2)$-colouring ensured by the claim. By design, $u$ gets colour $c$, and $v$ gets some colour that is not $c-1,c$ or $c+1$, and hence we have a $L$-$(7,2)$-colouring of $C_{n}$.
\end{proof}

There is a harder case we can deal with. 

\begin{definition}
A $4$-interval-list-assignment $L$ of $C_{n}$ is nearly safe if there exists a vertex $v$ with neighbours $v_{1},v_{2}$ where $L(v_{1}) = L(v_{2})$ and $L(v)$ shares at most $3$ colours with $L(v_{1})$. 
\end{definition}

\begin{prop}
Let $L$ be a $4$-interval-list-assignment of $C_{n}$ which is nearly safe. Then we can find a $L$-$(7,2)$-colouring.
\end{prop}

\begin{proof}
Let $v$ be a vertex with neighbours $v_{1}$ and $v_{2}$ where $L(v_{1}) = L(v_{2})$ and $L(v)$ shares at most $3$ colours with $L(v_{1})$. Without loss of generality suppose that $L(v) = \{0,1,2,3\}$. We consider cases.

 If $L(v)$ shares exactly $1$ colour with $L(v_{1})$, then either $L(v_{1}) = \{3,4,5,6\}$, or $L(v_{1}) = \{1,6,5,0\}$, which implies that $L$ is a safe list-assignment. This case follows from Observation \ref{safeclaim}. 

If $L(v)$ shares exactly $2$-colours with $L(v_{1})$, then $L(v_{1}) = \{2,3,4,5\}$ or $L(v_{1}) = \{1,0,6,5\}$. Again, either of these lists imply that $L$ is safe, and thus the result follows from Observation \ref{safeclaim}.

 Therefore we can assume that $L(v_{1}) = \{1,2,3,4\}$ or $L(v_{1}) = \{6,0,1,2\}$. Without loss of generality assume that $L(v_{1}) = \{1,2,3,4\}$. Now colour $v$ with $0$, and remove $1$ from the lists of $v_{1}$ and $v_{2}$, and delete $v$. The remaining path satisfies the conditions of Proposition \ref{sizetwoclaimsimplified}. Therefore we can find a $L$-$(7,2)$-colouring of the path which extends to an $L$-$(7,2)$-colouring of the entire graph, as desired.
\end{proof}

Finally we can cover the remaining non-uniform cases.

\begin{prop}
Let $L$ be a $4$-interval-list-assignment of $C_{2k+1}$ which is not uniform, not safe and not nearly safe. Then there is an $L$-$(7,2)$-colouring of $C_{2k+1}$.
\end{prop}

\begin{proof}
As $L$ is not uniform, let $vv_{1} \in E(C_{2k+1})$ such that $L(v) \neq L(v_{1})$. As $L$ is not safe, $L(v)$ and $L(v_{1})$ share three colours. Without loss of generality, assume that $L(v) = \{0,1,2,3\}$. Let $v_{2}$ be the other neighbour of $v$ that is not $v_{1}$. As $L$ is not nearly-safe, we can assume without loss of generality that $L(v_{1}) = \{6,0,1,2\}$ and $L(v_{2}) = \{1,2,3,4\}$. First suppose that $v,v_{1},v_{2}$ form a triangle. Then colour $v$ with $0$, $v_{1}$ with $2$ and $v_{2}$ with $4$.

 So we can assume we have at least five vertices. Let $v_{1,1}$ be the neighbour of $v_{1}$ that is not $v$. Initially suppose that $L(v_{1,1}) \neq \{1,2,3,4\}$ or $\{0,1,2,3\}$. Then colour $v$ with $0$ and $v_{1}$ with $2$. Then the set of available colours at $v_{2}$ has size $3$, and the set of available colours at $v_{1,1}$ is at least two. Therefore by Proposition \ref{sizetwoclaimsimplified} there is an $L$-$(7,2)$-colouring. 
  
  So we can assume that $L(v_{1,1}) = \{1,2,3,4\}$ or $\{0,1,2,3\}$. Regardless of these two lists, colour $v$ with $1$ and colour $v_{1}$ with $6$. Then the set of available colours for $v_{1,1}$ is at least $3$, and the set of available colours for $v_{2}$ is exactly $2$. Thus by Proposition \ref{sizetwoclaimsimplified} we have a $L$-$(7,2)$-colouring. 
\end{proof}

Now we observe that uniform lists do not admit an $L$-$(7,2)$-colouring of $C_{2k+1}$. 

\begin{obs}
\label{uniformlistsdonotwork}
Any uniform $4$-interval-list assignment $L$ of $C_{2k+1}$ does not admit a $L$-$(7,2)$-list colouring. 
\end{obs}

\begin{proof}
Without loss of generality we can assume that $L$ assigns the colours $0,1,2,$ and $3$ to each vertex. Suppose $f$ is an $L$-$(7,2)$-colouring of $C_{2k+1}$. Then the image of $f$ is a subgraph of the graph in $G_{7,2}$ induced on the vertices $0,1,2$ and $3$. However, this is bipartite, which by composing homomorphisms, would imply that $C_{2k+1}$ is bipartite, a contradiction. 
\end{proof}

Putting it all together, we have

\begin{lemma}
\label{oddcyclelists}
A $4$-interval-list-assignment $L$ of $C_{2k+1}$ admits an $L$-$(7,2)$-colouring if and only if $L$ is not uniform. 
\end{lemma}

Now we can prove the main result of this section.

\begin{thm}
Let $G$ be a $4$-critical graph with no $(7,2)$-colouring, and whose Gallai Tree contains an odd cycle. Then $G$ is isomorphic to an odd wheel. 
\end{thm}

\begin{proof}
Suppose we have an odd cycle $C$ in the Gallai Tree. By $4$-criticality, $G-C$ has a $3$-colouring, say $f$. As each vertex in $C$ has degree $3$ in $G$, each vertex in $C$ has exactly one neighbour not in $C$.
For every vertex $v \in V(C)$, let $v'$ be the neighbour of $v$ not in $C$. Assign to $v$ the list $N_{G_{7,2}}(f(v'))$. 

 This list assignment is $4$-interval, and hence by Theorem \ref{oddcyclelists}, we may assume this list assignment is uniform. Then for all $v \in V(C)$, we may assume that $f(v') =0$. If for every $u,v \in V(C)$, we have $u' = v'$, then $G$ is isomorphic to an odd wheel. Thus for some $u,v \in V(C)$, we have $u' \neq v'$. Change the colour of $u'$ to $6$. Observe that this is still a $(7,2)$-colouring. But now if we update the lists on $C$, we do not have a uniform list, and hence we can apply Theorem \ref{oddcyclelists} to extend the $(7,2)$-colouring of $G-C$ to a $(7,2)$-colouring of $G$, a contradiction. 
\end{proof}

We finish this section by observing that odd wheels are not counterexamples to Theorem \ref{maintheorem}.

\begin{obs}
When $k \geq 3$, the graph $W_{2k+2}$ has
\[e(W_{2k+2}) \geq \frac{17v(W_{2k+2})}{10}.\]
\end{obs}

\begin{proof}
Observe that $W_{2k+2}$ has $e(W_{2k+2}) = 4k+2$, and $v(W_{2k+2}) = 2k+2$. Then 
\[4k+2 \geq \frac{17(2k+2)}{10}\] 
which is equivalent to 
\[40k + 20 \geq 34k +34,\]
which simplifies to 
\[6k \geq 14\]
which is true if $k \geq 3$. 
\end{proof}

\section{A detour to $k$-critical graphs}
\label{detourtokcrit}
In this section we observe that we can extend the ideas of the previous section to prove that every $k$-critical graph with no $(2k-1,2)$-colouring has no block in the Gallai Tree isomorphic to a $K_{k-1}$.

\begin{lemma}
\label{colouringlemma}
Let $k \geq 2$. Let $L$ be a $k$-list-assignment of $K_{k+1}$. Then $K_{k+1}$ is $L$-colourable unless $L$ is uniform.
\end{lemma}

\begin{proof}
We proceed by induction on $k$. The greedy algorithm gives the result when $k = 2$, so assume that $k \geq 3$. Now suppose $L$ is not uniform. Let $uv \in E(K_{k+1})$ such that $L(u) \neq L(v)$. Let $c \in L(u) \setminus L(v)$. Now colour $u$ with $c$ and remove $c$ from the lists of the remaining vertices. As $c \not \in L(v)$, we may apply induction, so there is an $L$-colouring of $K_{k+1}-u$, and hence an $L$-colouring of $K_{k+1}$. 
\end{proof}
 
We will say a list assignment $L$ is \textit{$(2k-1,2)$-near-uniform} if all vertices receive one of two possible lists, and these lists correspond to the neighbourhoods of two non-adjacent vertices in $G_{2k-1,2}$.  Further we will assume a near-uniform list assignment is not uniform. 
\begin{lemma}
\label{listcolouring}
Let $k \geq 4$. Let $L$ be a $(2k-1,2)$-near-uniform list assignment of $K_{k-1}$. Then there is an $L$-$(2k-1,2)$-colouring of $K_{k-1}$.
\end{lemma}

\begin{proof}
If $k = 4$ this follows from Lemma \ref{oddcyclelists}. Therefore we proceed by induction and assume that $k \geq 5$. Without loss of generality, we may assume that the lists are the intervals $[2,2k-3]$ and $[3,2k-2]$. Colour some vertex with $2$. Then the new lists are $[4,2k-3]$, and $[4,2k-2]$. Viewing this new list assignment as a near uniform $(2(k-1)-1,2)$-list assignment of $K_{k-2}$, we see that they correspond to the neighbourhoods of the vertices $2$ and $3$, and hence by induction there is an $L$-$(2k-1,2)$-colouring of $K_{k-1}$. 
\end{proof}

\begin{cor}
Suppose $G$ is a $k$-critical graph which contains a $K_{k-1}$ where for every $v \in V(K_{k-1})$, $\deg(v) = k-1$. Then either $G$ is isomorphic to $K_{k}$, or $G$ has a $(2k-1,2)$-colouring.
\end{cor}

\begin{proof}
Suppose that $G$ contains a $K_{k-1}$ where every vertex in the $K_{k-1}$ has degree $k-1$. Let $f$ be a $k-1$ colouring of $G-K_{k-1}$, where we may assume that the $k-1$ colouring uses the colours $\{0,2,4,\ldots,2(k-2)\}$. By Lemma \ref{colouringlemma}, we can extend $f$ to a $k-1$-colouring of $G$ unless all vertices adjacent to the $K_{k-1}$ receive the same colour, which without loss of generality we may assume to be $0$. If there is only one such vertex, then $G$ is isomorphic to $K_{k}$. Thus there is at least two vertices. Change the colour of one of these vertices from $0$ to $2k-2$. This remains a $(2k-1,2)$-colouring, and now we can apply Lemma \ref{listcolouring} to extend the colouring, completing the claim. 
\end{proof}

\section{Acyclic Gallai Trees- A reduction to paths}
\label{reductiontopathssection}
 In this section we prove that if the Gallai Tree is acyclic, then every component of the Gallai Tree is isomorphic to a path or a claw.  We also show that a vertex minimal counterexample to Theorem \ref{maintheorem} has no claw component. For the rest of the paper, we will always assume that the Gallai Tree does not contain an odd cycle. The next lemma is the most important lemma in the entire paper. 

\begin{lemma}
\label{reconfiguringpathsoflength3}
Let $x,y,z \in V(G)$ such that $x,y$ and $z$ have degree $3$, and $xy,yz \in E(G)$. Let $x',x'',y',z',z''$ be the other neighbours of $x,y,z$ respectively. Then up to relabelling the vertex labels, $x' = z'$, $y'x'' \in E(G)$, and $y'z'' \in E(G)$. Additionally, $x'' \neq z''$. 

Further, if $x' \neq y'$, and there are two distinct vertices $x_{x',x''}$, $x_{x',z''}$ not in $\{x,y,z\}$ where $x_{x',x''}$ is adjacent to $x'$ and $x''$, and $x_{x',z''}$ is adjacent to $x'$ and $z''$.  
\end{lemma}

\begin{proof}

Let $f$ be a $3$-colouring of $G - \{x,y,z\}$. Without loss of generality we may assume that $f(y') =0$.
If $f(x') = f(x'')$, then simply give $z$ a colour from its available colours, then give $y$ an available colour, and finally as $f(x') = f(x'')$, $x$ has an available colour and we can extend the colouring. Hence $f(x') \neq f(x'')$, and similarly $f(z') \neq f(z'')$. If $\{f(x'),f(x'')\} = \{f(z'),f(z'')\}$, then give $x$ and $z$ the same colour, and we can extend this colouring to $y$. Finally, if $\{f(x'),f(x'')\} = \{2,4\}$, then colour $x$ with $0$, colour $z$ with any available colour, and we can extend the colouring to $y$. A similar argument works when $\{f(z'),f(z'')\} = \{2,4\}$.  

Thus without loss of generality we can assume that $f(x') = 0$, $f(x'') = 2$, $f(z') = 0$ and $f(z'') = 4$.  Observe this implies that $x'' \neq z''$ as they have different colours. 
 
\begin{claim}
$z' =x'$.
\end{claim}

\begin{proof}
If not, change the colour of $z'$ to $6$ and extend the colouring by colouring $x$ with $5$, $y$ with $3$, and $z$ with $1$. 
\end{proof}

\begin{claim}
$y'x'' \in E(G)$.
\end{claim}

\begin{proof}
Suppose not. Then change the colour of all vertices in $N_{4}(N_{2}(y'))$ to $5$, and change the colour of all vertices in $N_{2}(y')$ to $3$ and change the colour of $y'$ to $1$.
 If $z'' \in N_{4}(N_{2}(y'))$, then colour $z$ with $3$, $y$ with $6$ and $x$ with $4$.
Therefore $z'' \not \in N_{4}(N_{2}(y'))$. If $x' \neq y'$  colour $z$ with $2$, $y$ with $6$ and $x$ with $4$. If $x' = y'$, colour $z$ with $6$, $y$ with $3$ and $x$ with $6$. 
\end{proof}

\begin{claim}
$y'z'' \in E(G)$.
\end{claim}

\begin{proof}
Suppose not. Then change the colour of $z''$ to $5$ and the colour of $y'$ to $6$. Then colour $x$ with $4$, $y$ with $1$, and $z$ with $3$.  
\end{proof}

To finish the proof, suppose that $x' \neq y'$. If $z'' \not \in N_{4}(N_{2}(x'))$, then simply change the colour of all vertices in $N_{4}(N_{2}(x'))$ to $5$, change the colour of all vertices in $N_{2}(x')$ to $3$, and change the colour of $x'$ to $1$. Then colour $z$ with $6$, $y$ with $2$, and $x$ with $5$. Thus there is a vertex $x_{x',z''}$ which is adjacent to both $x'$ and $z''$. To see there is also a vertex $x_{x',x''}$ which is adjacent to $x'$ and $x''$, simply exchange the colours $2$ and $4$ on all vertices, and then repeat the above argument. Distinctness follows from the fact that their colours are different. 
\end{proof}

Observe that the vertices $x_{x',x''}$ and $x_{x',z''}$ may just be $x''$ and $z''$ if there are edges $x'x''$ and $x''z''$.  Now we can prove that the Gallai Tree of a $4$-critical graph with no $(7,2)$-colouring has every component isomorphic to either an odd cycle, a path, or a claw. 

\begin{cor}
\label{reductiontopaths}
Let $G$ be a $4$-critical graph with no $(7,2)$-colouring. If the Gallai Tree of $G$ is acyclic, then every component is either isomorphic to a path, or contains at most four vertices.
\end{cor}

\begin{proof}
Let $T$ be the Gallai Tree for $G$. If $T$ has no vertex of degree $3$ then every component of $T$ is a path and we are done. Let $T$ be a component which contains a vertex of degree $3$. Let $y$ be such a vertex, and let $x,z,y'$ be the three vertices adjacent to $y$ with degree $3$. 

Apply Lemma \ref{reconfiguringpathsoflength3} to $x,y,z$. Then $y'$ is adjacent to a neighbour of $x$, $x''$ and $y'$ is adjacent to a vertex of $z$, $z''$. Further $x$ and $z$ share a neighbour $x'$ that is not $y$. If $x'$ has degree $3$, then $x,y,z,x'$ is a cycle, contradicting that the Gallai Tree is acyclic. If $x''$ has degree $3$, then $y',x'',x,y$ is a cycle in the Gallai Tree - a contradiction. If $z''$ has degree $3$, then $y',z'',z,y$ is a cycle in the Gallai Tree - a contradiction. Hence $T$ is isomorphic to a claw, and the claim follows. 
\end{proof}

 We claim that in a vertex minimal counterexample, the Gallai Tree contains no component isomorphic to a claw.

\begin{lemma}
\label{noclaw}
Let $G$ be a $4$-critical graph with no $(7,2)$-colouring. Let $C$ be a claw component of the Gallai Tree, where $V(C) = \{x,y,z,y'\}$, where $y$ is adjacent to all of $x,z$ and $y'$. Let $G'$ be the graph obtained by identifying all of the vertices in $C$ into a single vertex and removing multiple edges and loops. Then $G'$ is $4$-critical and has no $(7,2)$-colouring. Further $e(G') = e(G) -6$ and $v(G') = v(G) -3$.
\end{lemma}

\begin{proof}

Let $x',x''$ and $z',z''$ be the neighbours of $x$ and $z$ not in $C$ respectively. Apply Lemma \ref{reconfiguringpathsoflength3} to $x,y,z$. Then we can assume up to relabelling the vertices that $x' = z'$, $y'$ is adjacent to $x''$ and $z''$. In the graph $G'$, let $w$ denote the vertex obtained after identifying $x,y,z,x'$. 

\begin{claim}
\label{is4colourable}
The graph $G'$ has a $4$-colouring.
\end{claim}

\begin{proof}
Take any $3$-colouring $f$ of $G - \{x,y,z,x'\}$. Then we can extend $f$ to a $4$-colouring of $G'$ by giving $w$ any available colour (there is an available colour as $w$ has degree $3$). 
\end{proof}

\begin{claim}
\label{no72}
The graph $G'$ has no $(7,2)$-colouring. 
\end{claim}

\begin{proof}
Suppose not, and let $f$ be a $(7,2)$-colouring of $G'$. Then consider the map $f'$ where for all vertices $v \in V(G) - \{x,y,z,x'\}$, $f'(v) = f(v)$, for all $t \in \{y',x,z\}$, $f'(t) = f(w)$, and let $f(y)$ be any colour in the set $N_{G_{7,2}}(f(w))$. This is a $(7,2)$-colouring of $G$, a contradiction.  
\end{proof}

\begin{claim}
The graph $G'$ is $4$-critical. 
\end{claim}

\begin{proof}
By Claim \ref{no72}, $G'$ does not have a $3$-colouring (as every $3$-colouring can be turned into a $(7,2)$-colouring), and by Claim \ref{is4colourable}, $G'$ is $4$-colourable, so $\chi(G') =4$. Therefore it suffices to show that $G'-e$ is $3$-colourable for all edges $e$. 

First consider deleting an edge incident to $w$ say $e$. To see that $G'-e$ has a $3$-colouring, simply take any $3$-colouring of $G- \{x,y,z,x'\}$, and there will be a colour left over for $w$, so we can extend the colouring (as $w$ has degree $3$).

Now consider an edge $e \in E(G')$ not incident to $w$. Then $G-e$ is $3$-colourable as $G$ is $4$-critical. Let $f$ be any $3$-colouring of $G-e$. If $f(y') = f(x)= f(z)$, then the colouring $f'$ where $f'(w) = f(y')$ and for all $v \in V(G') -w$, $f'(v) = f(v)$ is a $3$-colouring of $G'$. Thus at least two colours appear on $y',x,z$. Note that at most $2$ colours appear on $y',x,z$ as if all three colours appeared, then $y$ would not have a colour. Without loss of generality, suppose that $f(y') = f(x) = 0$, and $f(z) = 2$. Then $f(y) = 4$. Additionally $f(z') = 4$, and $f(x') =4$. But then $z$ is not adjacent to a vertex coloured $0$, and so we can change the colour of $z$ to $0$, and apply the case where $f(y') = f(x) =f(z)$. Hence $G'$ is $4$-critical. 
\end{proof}

Finally, one simply observes that $e(G') = e(G)-6$ and $v(G') = v(G) -3$. 
\end{proof}

\begin{cor}
In a vertex minimal counterexample to Theorem \ref{maintheorem}, all components of the Gallai Tree are paths.
\end{cor}

\begin{proof}
Let $G$ be a vertex minimal counterexample to Theorem \ref{maintheorem}. Then $G$ is not isomorphic to an odd wheel, so all components of the Gallai Tree of $G$ are either paths or claws. Suppose $G$ contains a claw component.  By Lemma \ref{noclaw}, there exists a graph $G'$ which is $4$-critical and has no $(7,2)$-colouring such that $e(G') = e(G) -6$, and $v(G') = v(G)-3$. First suppose that $G'$ is not $K_{4}$ or $W_{6}$. Then by minimality, we have that
\[e(G') \geq \frac{17v(G)}{10}.\]
Thus
\[e(G) -6 \geq \frac{17(v(G)-3)}{10}.\]
Rearranging we have
\[e(G) \geq \frac{17v(G) +9}{10}\]
contradicting that $G$ is a counterexample. 

Now assume that $G'$ is isomorphic to $K_{4}$. Then $v(G) =7$ and $e(G) = 12$, and clearly $12 \geq \frac{119}{10}$. Now assume that $G'$ is isomorphic to $W_{5}$. Then $v(G) = 9$ and $e(G) = 16$. Clearly $16 \geq \frac{153}{10}$. As these are all possibilities, the claim holds. 
\end{proof}

We will assume from here on, that all components of the Gallai Tree are paths.

\section{Structure around path components}
\label{remainingarguments}

The purpose of this section is to argue that the vertices near path components must have large degree. We start off with some easy observations. For any vertex $v$, let $N[v]$ denote the \textit{closed neighbourhood of $v$}, that is the set of vertices $N(v) \cup \{v\}$.

\begin{prop}
\label{closedneighbourhoods}
There does not exist two vertices $u$ and $v$ where $\deg(v) = \deg(u) = 3$,  and $N[v] = N[u]$. 
\end{prop}

\begin{proof}
Suppose not. Let $x$ and $y$ be the two neighbours of $u$ and $v$ which are not $u$ or $v$. Note $xy \not \in E(G)$, as otherwise $G$ contains a $K_{4}$, and hence is isomorphic to $K_{4}$.  By $4$-criticality, $G-u-v$ has a $3$-colouring, say $f$.  We consider cases. 

If $f(x) =f(y)$, then we can extend not only to a $(7,2)$-colouring, but a $3$-colouring, a contradiction. So we may assume that $f(x) \neq f(y)$. Without loss of generality, let $f(x) =0$.  Suppose $f(y) = 2$. Now consider $N_{0}(y)$. We change the colour of every vertex in $N_{0}(y)$ to $6$, and then change the colour of $y$ to $1$. Let $f'$ be the resulting $(7,2)$-colouring. Now we can extend $f'$ to a $(7,2)$-colouring of $G$ by letting $f'(u) = 3$ and $f'(v) = 5$. The rest of the cases follow by exchanging colours and applying one of the above arguments if necessary. 
\end{proof}

Now we prove the most important lemma in the section, which despite being very simple, enforces a large amount of local structure around path components.

\begin{lemma}
\label{doubleneighbourhoodreconfig}
 Let $v$ be a vertex where $N(v) = \{x,y,z\}$. For any pair $w,t \in \{x,y,z\}$, either $wt \in E(G)$, or there is a vertex $x_{w,t} \neq v$ where $wx_{w,t} \in E(G)$ and $x_{w,t}t \in E(G)$. If $wt \not \in E(G)$, then for any $w',t' \in \{x,y,z\}$, $x_{w,t} = x_{w',t'}$ if and only if $\{w,t\} = \{w',t'\}$. 
\end{lemma}

\begin{proof}
Suppose without loss of generality that $xy \not \in E(G)$, and $xy$ does not lie in a $4$-cycle with $v$. Then by permuting colours if necessary, there is a $3$-colouring of $G-v$ such that $f(x) =0$, $f(y) =4$, and $f(z) =2$. Now change all vertices colours in $N_{4}(N_{2}(x))$ to $5$,  change all vertices colours in $N_{2}(x)$ to $3$ and change the colour of $x$ to $1$. Observe that the colour of $y$ did not change. Then colour $v$ with $6$ to obtain a $(7,2)$-colouring of $G$, a contradiction. Hence either $xy \in E(G)$, or there is a vertex in $N_{2}(x)$ which is adjacent to $y$, as desired. Uniqueness comes from the fact that we can assume the vertex in a $4$-cycle with $v,x,y$ is coloured $2$, and for any pair that we apply this argument to, we get a distinct colour, and hence the vertices are distinct. 
\end{proof}

We observe that if $x_{w,t}$ exists, it  may in fact be one of $\{x,y,z\}$. However if say $x_{x,z} =y$, then $xy \in E(G)$ and $yz \in E(G)$. 

\begin{obs}
\label{fancut}
Let $v$ be a vertex with $N(v) = \{x,y,z\}$.  Suppose that $G[\{x,y,z\}]$ contains at least two edges, with $t \in \{x,y,z\}$ having degree $2$ in $G[\{x,y,z\}]$. Then $G[\{x,y,z\}]$ contains exactly two edges, and for $w,r \in \{x,y,z\} - \{t\}$, $N(w) \cap N(r) = \{v,t\}$. 
\end{obs}

\begin{proof}
Suppose not. Observe that $x,y$ and $z$ cannot induce a triangle as then we have a clique cutset in a $4$-critical graph. Thus without loss of generality, let $xy,yz \in E(G)$. Suppose there is a vertex $x_{x,z} \not \in \{v,y\}$ such that $xx_{x,z} \in E(G)$ and $zx_{x,z} \in E(G)$. 

Now let $f$ be a $3$-colouring of $G - \{xx_{x,z}\}$. Then $f(x) = f(x_{x,z})$ otherwise we have a $3$-colouring of $G$. Without loss of generality we may assume that $f(x) = 0$. Then $f(v) \neq 0$, so without loss of generality $f(v) =2$. But then $f(y) = f(z) = 4$, a contradiction as $yz \in E(G)$.  
\end{proof}

\subsection{The case where $x,y,z$ induces an edge}

For the subsection, we have a vertex $v$ with neighbours $x,y,z$ and we will assume that $xy \in E(G)$, $yz,xz \not \in E(G)$. Thus by Lemma \ref{doubleneighbourhoodreconfig} there are distinct vertices $x_{y,z}$ and $x_{x,z}$ where $x_{y,z}$ is adjacent to both $y$ and $z$, and $x_{x,z}$ is adjacent to both $x$ and $z$. Further $x_{y,z}$ and $x_{x,z}$ are not in $\{x,y,z\}$.  The goal of the subsection is to show that the neighbours of $v$ have large degree. We will prove stronger claims than what is necessary to deduce Theorem \ref{maintheorem}, but we believe the additional claims would be useful if trying to improve the bound on Theorem \ref{maintheorem}. We start by proving $\deg(x) \geq 4$ and $\deg(y) \geq 4$. The following observation is well known. 

\begin{obs}
\label{3colouringextension}
Let $w$ and $t$ be adjacent vertices of degree $3$. Suppose that $w$ and $t$ have a common neighbour $a$. Suppose that the other neighbour of $w$ is $b$, and the other neighbour  of $t$ is $c$. If there is a $3$-colouring $f$ of $G-w-t$ such that $f(b) \neq f(c)$, then $G$ has a $3$-colouring. In particular, $bc \not \in E(G)$.
\end{obs}

\begin{proof}
Let $f$ be a $3$-colouring of $G-w-t$ so that $f(b) \neq f(c)$. Without loss of generality, suppose that $f(b) =0$ and $f(c) =2$. If $f(a) =2$, then colour $w$ with $4$ and $t$ with $0$. If $f(a) =0$, colour $w$ with $2$ and $t$ with $4$. If $f(a) =4$, then colour $w$ with $2$ and $t$ with $0$. In all cases, we get a $3$-colouring of $G$. 

To see that $bc \not \in E(G)$, if $bc \in E(G)$, then every $3$-colouring of $G-w-t$ has $f(b) \neq f(c)$, and extends to a $3$-colouring of $G$, contradicting $4$-criticality.
\end{proof}

\begin{lemma}
\label{keywheelcutlemma}
Let $w$ and $t$ be two vertices in $G$ both having degree $3$. Suppose that $w$ and $t$ share a common neighbour $a$. Suppose that $b$ is the other neighbour of $w$, and $c$ is the other neighbour of $t$. Then $ab \in E(G)$, and $ac \in E(G)$. 
\end{lemma}

\begin{proof}
Suppose not. Without loss of generality, we can assume that $ab \not \in E(G)$. Let $f$ be a $3$-colouring of $G-w-t$. Without loss of generality, we can assume that $f(b) =0$, which by Observation \ref{3colouringextension}, implies that $f(c) = 0$. If $f(a) =0$, then colouring $w$ with $2$ and $t$ with $4$ is a $3$-colouring of $G$.

 So without loss of generality assume that $f(a) = 2$. Now consider $N_{2}(b)$, and $N_{4}(N_{2}(b))$. Change the colour of all vertices in $N_{4}(N_{2}(b))$ to $5$, and change the colour of all the vertices in $N_{2}(b)$ to $3$, and finally change the colour of $b$ to $1$. Now as $ab \not \in E(G)$, we can now colour $w$ with $6$ and $t$ with $4$, contradicting that $G$ has no $(7,2)$-colouring. Thus $ab \in E(G)$, and by the same argument, we have that $ac \in E(G)$.   
\end{proof}

\begin{cor}
\label{degreefourclaim}
Both $\deg(x) \geq 4$ and $\deg(y) \geq 4$. 
\end{cor}

\begin{proof}
Suppose towards a contradiction that $y$ has degree $3$. Then by Lemma \ref{keywheelcutlemma}, $xz \in E(G)$.
But we assumed at the start of the section, that the vertices $x,y,z$ induce exactly one edge, and we now have edges $xy$ and $xz$, a contradiction. 
\end{proof}

Now we make a straightforward observation. 

\begin{obs}
Both $yx_{x,z} \not \in E(G)$, and $xx_{y,z} \not \in E(G)$.
\end{obs}

\begin{proof}
Suppose that $yx_{x,z} \in E(G)$. Then $N(v) \subseteq N(x_{x,z})$, which does not occur in a $4$-critical graph, a contradiction. An analogous argument works for $xx_{y,z}$. 
\end{proof}

\begin{obs}
If the Gallai Tree of $G$ has no claw component, then one of $z,x_{x,z}$ or $x_{y,z}$ has degree at least $4$.
\end{obs}

\begin{proof}
If not, then $v,z,x_{x,z}$ and $x_{y,z}$ form a claw in the Gallai Tree. 
\end{proof}

Now we want to understand what happens when $x$ and $y$ share a neighbour that is not $v$. 

\begin{lemma}
\label{commonneighbour}
Suppose $x$ and $y$ have a common neighbour $w$ that is not $v$. Then at least one of the following occurs:
\begin{itemize}
\item{There is at least one $t \in \{x,y\}$ such that $\deg(t) \geq 5$.}
\item{There is at least one $t \in \{w,z\}$ such that $\deg(t) \geq 4$.}
\end{itemize}
\end{lemma}

\begin{proof}
Suppose none of the above conditions occur. This implies that  $\deg(x) = 4$ and $\deg(y)=4$. Let $f$ be a $3$-colouring of $G-\{v,x,y\}$. Without loss of generality suppose that $f(w) = 0$.

\textbf{Case 1: Either $f(x_{x,z}) = 0$ or $f(x_{y,z}) = 0$}

Without loss of generality suppose that $f(x_{x,z}) =0$. Then $f(z) \neq 0$. Colour $v$ with $0$. There exists at least one available colour for $y$, so colour $y$ with this colour, and then the neighbourhood of $x$ sees at most two colours, and so there is a colour available for $x$, thus we get a $3$-colouring of $G$. A similar argument works when $f(x_{y,z}) =0$. 

\textbf{Case 2: $f(x_{x,z}) = 2$ and $f(x_{y,z}) =4$}

Observe in this case that $f(z) = 0$. We claim that either $wx_{x,z} \in E(G)$, or there is a vertex $x_{w,x_{y,z}}$ coloured $2$ adjacent to both $w$ and $x_{y,z}$. If not, change the colour of all vertices in $N_{4}(N_{2}(w))$ to $5$, change the colour of all vertices in $N_{2}(w)$ to $3$, and change the colour of $w$ to $1$. Then colour $x$ with $4$, $y$ with $6$, and $v$ with $2$.

Now we claim that $w$ is adjacent to a vertex coloured $4$. If not, change the colour of $w$ to $4$. Then colour $x$ with $0$, $y$ with $2$, and $v$ with $4$. From this, we deduce that $w$ is adjacent to a vertex coloured $4$ and a vertex coloured $2$, and hence $\deg(w) \geq 4$. 

\textbf{Case 3: $f(x_{x,z})=2$ and $f(x_{y,z}) =2$}

Suppose $f(z) = 4$. In this case we claim that both $wx_{x,z} \in E(G)$ and $wx_{y,z} \in E(G)$. Suppose without loss of generality that $wx_{x,z} \not \in E(G)$. Then change the colour of all vertices in $N_{4}(N_{2}(w))$ to $5$, change the colour of all vertices in $N_{2}(w)$ to $3$, and change the colour of $w$ to $1$. Then colour $x$ with $4$, $y$ with $6$ and $v$ with $2$. Thus in this case $\deg(w) \geq 4$ (in fact, if this case occurs then the graph is isomorphic to a $C_{6}$-expansion of $K_{4}$)

Now suppose that $f(z) = 0$. If $\deg(z) = 3$, then $z$ is not adjacent to a vertex coloured $4$. Hence we can change the colour of $z$ to $4$, and apply the above argument. Thus $\deg(z) \geq 4$.

\end{proof}

Now we will want to understand what happens when $x$ and $y$ do not share a neighbour and have small degree. 
\begin{lemma}
\label{pivotalcommonneighbourlemma}
Suppose  $x$ and $y$ do not share a common neighbour other than $v$. Let $x'$ and $y'$ be the neighbours of $x$ and $y$ that are not $x_{x,z}$ and $x_{y,z}$. Then at least one of the following occurs.

\begin{itemize}
\item{There exists a $t \in \{x,y,z\}$ such that $\deg(t) \geq 5$.}
\item{The edge $x'y' \in E(G)$,  at most one of $x'$ or $y'$ have degree $3$, and either $\deg(z) \geq 4$, or both $x_{x,z}$ and $x_{y,z}$ has degree $4$.}
\item{The edge $x'y' \not \in E(G)$, there exists a $t \in \{x_{x,z},y'\}$ such that $\deg(t) \geq 4$ and a $w \in \{x_{y,z},x'\}$ such that $\deg(w) \geq 4$.}
\item{The edge $x'y' \not \in E(G)$, $\deg(z) \geq 4$, and at least one of $x_{y,z}$ or $x_{x,z}$ has degree $4$.}
\end{itemize}

\end{lemma}

\begin{proof}
Suppose none of the conditions hold. In particular this implies that $\deg(x) = \deg(y) = 4$, and $\deg(z) \leq 4$. 

Consider a $3$-colouring of $G- \{v,x,y\}$. Without loss of generality we may assume that $f(x') = 0$. We consider cases.

\textbf{Case 1: $f(y') = 0$}

Observe that if this occurs, then $x'y' \not \in E(G)$, as $f(x') = f(y')$. 

\textbf{Subcase 1: $f(x_{x,z}) = f(x_{y,z})$}

 If $f(x_{x,z}) = 0$, then as $zx_{x,z} \in E(G)$, $f(z) \neq f(x_{x,z})$. Thus colour $x$ and $y$ with $2$ and $4$, and colour $v$ with $0$. Therefore $f(x_{x,z}) \neq 0$.
 
 Now suppose that $f(x_{x,z}) = 2$. We claim that $x'$ and $y'$ are adjacent to a vertex coloured $4$. Suppose not and without loss of generality suppose that $x'$ has no neighbours coloured $4$. Change the colour of $x'$ to $4$. Then colour $x$ with $0$, $y$ with $4$, and since $f(z) \neq 2$, there is an available colour for $v$, a contradiction.
 
 \textbf{Subsubcase 1: $f(z) = 4$}

  We claim that $x'x_{y,z} \in E(G)$. If not, change the colour of all vertices in $N_{4}(N_{2}(x'))$ to $5$, change the colour of all vertices in $N_{2}(x')$ to $3$, and the colour of $x'$ to $1$. Then colour $x$ with $6$, $y$ with $4$, and $v$ with $2$. By an analogous argument, $y'x_{x,z} \in E(G)$. 
 
 Now we claim that either $x_{x,z}$ is adjacent to a vertex coloured $0$ that is not $y'$, or $y'$ is adjacent to a vertex coloured $2$ that is not $x_{x,z}$. If not, exchange the colours of $x_{x,z}$ and $y'$. Then colour $y$ with $4$, $x$ with $2$, and $v$ with $0$. Thus either $\deg(x_{x,z}) \geq 4$, or $\deg(y') \geq 4$. By an analogous argument, either $\deg(x_{y,z}) \geq 4$ or $\deg(x') \geq 4$, a contradiction. 
 
\textbf{Subsubcase 2: $f(z) = 0$}

Observe that if $z$ is not adjacent to a vertex coloured $4$, then we can simply change the colour of $z$ to $4$ and apply the above case analysis to conclude there is a $t \in \{x_{x,z},y'\}$ such that $\deg(t) \geq 4$, and a $w \in \{x_{y,z},x'\}$ where $\deg(w) \geq 4$. So we can assume that $z$ is adjacent to a vertex coloured four, and hence $\deg(z) \geq 4$.  

We claim that both $x_{x,z}$ and $x_{y,z}$ are adjacent to a vertex coloured $4$. Suppose $x_{x,z}$ is not adjacent to a vertex coloured $4$. Then change the colour of $x_{x,z}$ to $4$ and the colour of $x'$ to $6$. Then colour $x$ with $1$, $y$ with $5$ and $v$ with $3$. Hence both $x_{x,z}$ and $x_{y,z}$ are adjacent to a vertex coloured $4$. 

Now consider the graph induced by the colour classes $0$ and $2$. Let $C$ be the component of this graph containing $z$. If this component is only $z,x_{x,z}$ and $x_{y,z}$, then colour $x_{x,z}$ and $x_{y,z}$ with $0$ and $z$ with $2$. Then by a previous case, we obtain a $3$-colouring. Thus either $z$ is adjacent to a vertex coloured $2$ that is not $x_{x,z}$ and $x_{y,z}$, or one of $x_{x,z}$ and $x_{y,z}$ is adjacent to a vertex coloured $0$ that is not $z$. In the case $z$ is adjacent to a vertex coloured $0$ that is not $z$, then $\deg(z) \geq 5$. Otherwise, at least one of $x_{y,z}$ or $x_{x,z}$ has degree $4$.

We do not consider the case where $f(x_{x,z}) = 4$ as it follows a similar analysis as above.

\textbf{Subcase 2: $f(x_{x,z}) \neq f(x_{y,z})$}

 First suppose that $f(x_{x,z}) = 0$. If $f(x_{y,z}) = 2$, then $f(z) = 4$, and we can extend to a $3$-colouring by colouring $x$ with $2$, $y$ with $4$, and $v$ with $0$. A similar argument works if $f(x_{y,z}) = 4$. Additionally, similar arguments work if $f(x_{y,z}) = 0$.
 
  Thus without loss of generality $f(x_{x,z}) = 2$ and $f(x_{y,z}) = 4$. Thus $f(z) = 0$. In this case, change the colour of $N_{4}(N_{2}(x'))$ to $5$, $N_{2}(x')$ to $3$, and $x'$ to $1$. Then colour $y$ with $2$, $x$ with $6$, and $v$ with $4$. 

\textbf{Case 2: $f(y') = 2$} 

\textbf{Subcase 1: $f(x_{x,z}) = f(x_{y,z})$}

 Suppose that $f(x_{x,z}) = 0$. Then $f(z) \in \{2,4\}$. Colour $y$ with $4$ and $x$ with $2$, and colour $v$ any available colour. A similar argument shows that if $f(x_{x,y}) = 2$, we can always extend to a $3$-colouring. Therefore we can assume that $f(x_{x,z}) = 4$. Observe that $f(z) \neq 4$, and hence colour $v$ with $4$, $x$ with $2$ and $y$ with $0$, a contradiction.

\textbf{Subcase 2: $f(x_{x,z}) \neq f(x_{y,z})$}

 First suppose $f(x_{x,z}) = 0$. If $f(x_{y,z}) =2$, then $f(z) = 4$, and colour $y$ with $4$, $x$ with $2$, and $v$ with $0$. A similar colouring works when $f(x_{y,z}) = 4$. Thus $f(x_{x,z}) \neq 0$, and similarly we can assume that $f(x_{y,z}) \neq 2$. 

Now suppose $f(x_{x,z}) = 2$. If $f(x_{y,z}) = 4$, then colour $x$ with $2$, $y$ with $0$ and $v$ with $4$. Hence, $f(x_{y,z}) =0$ and thus $f(z) = 4$. Now suppose that $x'y' \not \in E(G)$. In this case, change the colour of all vertices in $N_{4}(N_{2}(x'))$ to $5$, change the colour of all vertices in $N_{2}(x')$ to $3$ and $x'$ to $1$. Then colour $x$ with $6$, $y$ with $4$, and $v$ with $2$, a contradiction. Thus $x'y' \in E(G)$.

 Similarly, if $x_{x,z}x_{y,z} \not \in E(G)$, then we change the colour of all vertices in  $N_{4}(N_{2}(x_{x,z}))$ to $5$,  the colour of all vertices in $N_{2}(x_{x,z})$ to $3$, and the colour of $x_{x,z}$ to $1$. Then colour $x$ with $4$, $y$ with $6$ and $v$ with $0$.

Now we claim that both $x'$ and $y'$ are adjacent to a vertex coloured $4$. If either $x'$ or $y'$ is not adjacent to a vertex of degree $4$, simply change one of the vertices to colour $4$, and then extend to a $3$-colouring using the same analysis as before.

Now we claim that either $x'$ is adjacent to a vertex coloured $2$ which is not $y'$, or $y'$ is adjacent to a vertex coloured $0$ which is not $x'$. If not, then simply exchange the colours on $x$ and $y$. But now we can extend to a $3$-colouring, a contradiction.

Thus it follows that at least one of $x'$ or $y'$ has degree $4$, and $x'y' \in E(G)$. 

Observe that as $x_{x,z},x_{y,z}$ and $z$ induce a triangle. If $\deg(z) = 3$, then note that $G[\{z,x_{x,z},x_{y,z}\}]$ induces exactly one edge, and then $\deg(x_{x,z}) \geq 4$ and $\deg(x_{y,z}) \geq 4$ by Observation \ref{degreefourclaim}. Otherwise $\deg(z) \geq 4$.

Lastly $f(x_{x,z}) = 4$. Then $f(x_{y,z}) = 0$ otherwise we use a previous case. Thus $f(z) = 2$. Then we can extend to a $3$-colouring with $x$ coloured $2$, $y$ coloured $4$, and $v$ coloured $0$. 
\end{proof}

\subsection{Long paths in the Gallai Tree}

We need to gain more understanding of long paths in the Gallai Tree. We start off with a Corollary of Lemma \ref{reconfiguringpathsoflength3}.

\begin{cor}
\label{alternatingpath}
Let $P$ be a path with at least three vertices where all vertices in $P$ have degree $3$. Let $V(P) = \{v_{0},\ldots,v_{n}\}$ and $E(P) = \{v_{i}v_{i+1} \, | \, i \in \{1,\ldots,n-1\}\}$. Let $v'_{0},v''_{0}$, $v'_{n},v''_{n}$ be the neighbours of $v_{0}$ and $v_{n}$ which are not in $P$, and let $v_{1}'$ be the neighbour of $v_{1}$ not in $P$. Then there is a $w \in \{v'_{0},v''_{0}\}$ and a $t \in \{v'_{n},v''_{n}\}$ such that given a bipartition $(A,B)$ of $P \cup \{w,t\}$ where $v_{0} \in A$ and $v_{1} \in B$, all vertices in $B$ are adjacent to $v'_{1}$, and all vertices in $A$ are adjacent to the vertex $w'$ in $\{v'_{0},v''_{0}\} - w$.

Further, for any $q \in \{w',v_{1}'\}$ and any $p \in \{w,t\}$, either $pq \in E(G)$, or there is a vertex $x_{p,q}$ such that $x_{p,q}$ is adjacent to both $p$ and $q$, and does not lie on $P$.   
\end{cor}

\begin{proof}
We proceed by induction on $n$. If $n =3$, the result follows from Lemma \ref{reconfiguringpathsoflength3}. Now assume $n \geq 4$. Consider the path $v_{1},\ldots,v_{n}$. Let $v_{0},v'_{1}$ be the vertices adjacent to $v_{1}$ not in $v_{1},\ldots,v_{n}$, and $v_{n}',v''_{n}$ be the vertices adjacent to $v_{n}$ not in $P$. Let $v_{2}'$ be the vertex not in $P$ adjacent to $v_{2}$. Apply the induction hypothesis to $v_{1},\ldots,v_{n}$. 

Observe that $v_{0}$ has degree $3$, so $v_{0}$ is not adjacent to any vertex of degree $3$ in $P$ aside from $v_{1}$, as then we would have a cycle of degree $3$ vertices. Thus when applying the induction hypothesis, we can conclude that $v_{0} =w$ (where $w$ is defined as in the statement). Let $(A,B)$ be a bipartition of $P \cup \{v_{0},v_{n}'\}$ (up to relabelling $v_{n}'$ with $v_{n}''$ if necessary), such that $v_{0} \in A$. Then by induction, $v_{1}'$ is adjacent to all vertices in $B$, and $v_{2}'$ is adjacent to all vertices in $A$. 

Now let $v_{0}',v_{2}'$ be the vertices adjacent to $v_{0}$ which are not $v_{1}$. Now apply Lemma \ref{reconfiguringpathsoflength3} to $v_{0},v_{1},v_{2}$, where in the context of that lemma statement, $v_{0} =x$, $v_{1}=y$ and $v_{2} =z$. Then $y' = v_{1}'$. As $v_{2}'$ is adjacent to both $v_{2}$ and $v_{0}$, $v_{2}' = x'$ in the lemma statement, and hence $v_{0}' = x''$ and so $v_{1}'v_{0}' \in E(G)$. 

Finally, observe that if $v_{2}'v_{0}' \not \in E(G)$, then Lemma \ref{reconfiguringpathsoflength3} ensures that there is a vertex $x_{v_{2}'v_{0}'}$ not in $v_{0},v_{1},v_{2}$ which is adjacent to both $v_{2}',v_{0}'$. Observe that $x_{v_{2}',v_{0}'}$ is not in $P$, as all vertices in $P$ have degree $3$. Similarly, if $v_{1}'v_{n}' \not \in E(G)$, then by induction we have a vertex $x_{v_{1}',v_{n}'}$ adjacent to vertices $v_{1}',v_{n}'$ and does not belong to the path $v_{1},\ldots,v_{n}$, and $x_{v_{1}',v_{n}'}$ is not $v_{0}$, as $v_{0}$ has degree $3$ ($v_{0}$ would be adjacent to $v_{1}',v_{1},v_{2}',v_{n}'$, and as $v_{1}'v_{n}' \not \in E(G)$, $v_{1}' \neq v_{2}'$). This completes the claim. 
\end{proof}

We can strengthen Lemma \ref{reconfiguringpathsoflength3} when the path of length three is a component of the Gallai Tree and a specific outcome occurs. 

\begin{lemma}
\label{pathoflength3kempe}
Suppose the following graph $H$ is a subgraph of $G$. Let $V(H) = \{x,y,z,x',x'',y', \\ z'',x_{x',x''},x_{x',z''}\}$. Let $E(H) = \{xy,xx',xx'',yz,yy',zx',zz'',x''x_{x',x''}, x''y', x_{x',x''}x', x'x_{x',z''}, \\ x_{x',z''}z'',y'z''\}$. 

Further suppose that all of $x,y$ and $z$ have degree three in $G$. Then at least one of the following occurs:

\begin{itemize}
\item{There exists a $t \in \{y',x',x'',z''\}$ such that $\deg(t) \geq 5$}
\item{There exists a $t \in \{x_{x',x''},x_{x',z''}\}$ such that $\deg(t) \geq 4$}
\end{itemize}
\end{lemma}

\begin{proof}

Suppose none of the conditions hold. 

Let $f$ be a $3$-colouring of $G- \{x,y,z\}$. Without loss of generality suppose that $f(x') = 0$. First suppose that $f(x'') =0$. Then colour $z$ with any available colour, $y$ with any available colour, and since $f(x') = 0$ and $f(x'') =0$, there is an available colour for $x$, a contradiction. A similar argument holds for $f(z'')$. 

If $f(x'') = f(z'')$, then colour $x$ and $z$ the same colour, and there is an available colour for $y$. Thus we can assume without loss of generality that $f(x'') = 2$ and $f(z'') =4$. Then $f(y') = 0$, $f(x_{x',x''}) = 2$ and $f(x_{x',z''}) =4$. 

Thus observe that $x',x_{x',z''},z'',z',x'',x_{x',x''}$ are a six cycle where for any three consecutive vertices, all three colours appear. Let $a,b \in \{x',x_{x',z''},z'',z',x'',x_{x',x''}\}$ such that $ab$ is an edge. We claim that either $a$ is adjacent to a vertex not $b$ with the same colour as $b$, or $b$ is adjacent to a vertex not $a$ with the same colour as $a$.  To see this, if not simply exchange the colours of $a$ and $b$, and by the previous case analysis (swapping colours if necessary), we can extend the colouring. 

Observe if $x'$ is adjacent to a vertex coloured $2$ or $4$ that is not $x_{x',z''}$ or $x_{x',x''}$ then $\deg(x') \geq 5$ and we are done. 

This implies that $x_{x',z''}$ is adjacent to a vertex coloured $0$ that is not $x'$, and thus either $\deg(x',z'') \geq 4$, or $z''$ is adjacent to a vertex coloured $2$ that is not $x_{x',z''}$. Similarly, this implies that either $\deg(z'') \geq 5$, or $z'$ is adjacent to a vertex coloured $4$ which is not $z''$. Continuing, this implies that either $\deg(z') \geq 5$ or $x''$ is adjacent to a vertex coloured $0$ that is not $z'$. Finally, this implies that either $\deg(x'') \geq 5$, or $\deg(x_{x',x''}) \geq 4$, a contradiction in either case, and so we conclude the claim. 
\end{proof}

We note that you can strengthen the above even further, but to the best of my knowledge you still cannot get significant improvements to the bound without some further arguments which are not clear to me.

\begin{lemma}
\label{pathoflength4kempe}
Suppose the following graph $H$ is a subgraph of $G$. Let $V(H) = \{x_{1},x_{2},x_{3},x_{4}, \\ x_{1}',x_{4}',u,v,x_{x_{1}',u},x_{v,x_{4}'}\}$, and $E(H) = \{x_{1}x_{1}',x_{1}u,x_{1}x_{2},x_{2}x_{3},x_{2}v,x_{3}u,x_{3}x_{4},x_{4}x_{4}', x_{4}v, x_{4}'u, \\ x_{4}'x_{x_{4}',v},x_{1}'v,x'_{1}x_{x_{1}',u}, vx_{x_{4}',v},x_{x_{1}',u}u\}$.

Further suppose that $x_{1},x_{2},x_{3},x_{4}$ all have degree $3$ in $G$. Then at least one of the following occurs:

\begin{itemize}
\item{There is a $t \in \{u,v\}$ such that $\deg(t) \geq 5$}
\item{There is a $t \in \{x_{1}',x_{4}'\}$ such that $\deg(t) \geq 5$ and there is a $w \in \{x_{x_{1}',u},x_{x_{4}',v}\}$ such that $\deg(w) \geq 4$}
\item{Both $\deg(x_{1}') \geq 5$ and $\deg(x_{4}') \geq 5$}
\item{Both $\deg(x_{x_{1}',u)} \geq 4$ and $\deg(x_{x_{4}',v)} \geq 4$}
\end{itemize}
\end{lemma}

\begin{proof}
We may assume none of the conditions holds. Let $f$ be a $3$-colouring of $G - \{x_{1},x_{2},x_{3},x_{4}\}$. Without loss of generality we may assume that $f(v) = 0$. We claim that $f(u) = 0$. Suppose not, and without loss of generality suppose that $f(u) =2$. If $f(x_{1}') = f(x_{4}')$, then $f(x_{1}') = 4$, and colour $x_{1}$ and $x_{3}$ with $0$, and $x_{2}$ and $x_{4}$ with $2$. If $f(x_{1}') = 2$ and $f(x_{4}') = 0$, then again colour $x_{1}$ and $x_{3}$ zero, and $x_{2}$ and $x_{4}$ two. If $f(x_{1}') =4$, then colour $x_{1}$ and $x_{3}$ zero, $x_{2}$ with $2$, and $x_{4}$ with any available colour. All other cases follow similarly.

Hence we can assume that $f(u) = 0$. This implies that $f(x'_{1}) \neq 0$ and $f(x'_{4}) \neq 0$, as they are adjacent to at least one of $u$ or $v$. Now we claim that $f(x'_{1}) = f(x_{4}') \in \{2,4\}$. Suppose $f(x'_{1}) \neq f(x'_{4})$. Then without loss of generality $f(x'_{1}) = 2$ and $f(x'_{4}) = 4$. Then colour $x_{1}$ and $x_{3}$ with $4$, and $x_{2}$ and $x_{4}$ with $2$.  Thus it follows that $f(x'_{1}) = f(x_{4}') \in \{2,4\}$. Without loss of generality we will assume that $f(x'_{1}) = 4$. Hence $f(x_{x'_{1},u}) = f(x_{v,x_{4}'}) = 2$.

 Now the proof is analogous to Lemma \ref{pathoflength3kempe}. We observe that for any adjacent pair of vertices $w,t \in \{x_{1}',x_{4}',v,u,x_{x_{1}',u},x_{v,x_{4}'}\}$, either $w$ is adjacent to a vertex coloured $f(t)$ that is not $t$, or $t$ is adjacent to a vertex coloured $f(w)$ which is not $w$. If not, we simply exchange the colours of $w$ and $t$, and remain a $3$-colouring, and extend this colouring to a $(7,2)$-colouring of $G$ (here this extension is possible by the same case analysis done above). 
 
 As we assumed that none of the conditions hold, this implies that all of $x_{x_{4}',u}$, $x_{4}'$, $x_{x_{1}',v}$, and $x_{1}'$ are adjacent to a vertex coloured $0$ not in $H$ as otherwise either $\deg(u) \geq 5$ or $\deg(v) \geq 5$. But now the exchanges possible on $x_{x_{4}',u}$ and $x_{4}'$ imply that either $\deg(x_{4}') \geq 5$ or $\deg(x_{x_{4}',u}) \geq 4$ and similarly for the pair $x_{1}'$ or $x_{x_{1}',v}$, and thus one of the conditions holds. 
\end{proof}

\begin{lemma}
\label{pathoflength5kempe}
Suppose that $G$ contains the following graph $H$ as a subgraph. Let $V(H) = \{x_{1},x_{2},x_{3},x_{4},x_{5},x_{1}',x_{5}',u,v,x_{x_{1}',v},x_{v,x_{5}'}\}$ and $E(H) = \{x_{1}x_{2},x_{1}x_{1}',x_{1}v,x_{2}u,x_{2}x_{3},  x_{3}v,x_{3}x_{4}, \\ x_{4}u,x_{4}x_{5},x_{5}x_{5}',x_{5}v, x_{5}'u,x_{5}'x_{v,x_{5}'},vx_{x_{5}',v},vx_{x_{1}',v}
,x_{1}'x_{x_{1}',v},x_{1}'u\}$.
Then at least one of the following occurs:
\begin{itemize}
\item{There exists a $t \in \{u,v,x_{1}',x_{5}'\}$ such that $\deg(t) \geq 5$}
\item{There exists a $t \in \{x_{x_{1}',v},x_{v,x_{5}'}\}$ such that $\deg(t) \geq 4$}
\item{$\deg(v) \geq 6$}
\end{itemize}
\end{lemma}

\begin{proof}
Let $f$ be a $3$-colouring of $G- \{x_{1},x_{2},x_{3},x_{4},x_{5}\}$. Without loss of generality we assume that $f(u) = 0$. We claim that $f(v) = 0$. Suppose  $f(v) = 2$. Then colour $x_{1},x_{3}$ and $x_{5}$ zero, and $x_{2}$ and $x_{4}$ two. A similar argument applies when $f(v) = 4$. Hence $f(v) =0$.

Now we claim that $f(x_{1}') \neq f(x_{4}')$. Suppose $f(x_{1}') = f(x_{2}')$. Then without loss of generality $f(x_{1}') = 2$. Then colour $x_{1},x_{3},x_{5}$ with $4$, and $x_{2}$ and $x_{4}$ with $2$. Thus without loss of generality we can assume that $f(x_{1}') = 4$ and $f(x_{4}') = 2$. Hence $f(x_{x_{1}',v}) = 2$ and $f(x_{v,x_{5}'}) = 4$. 

Then observe that for any pair of adjacent vertices $w,t \in \{u,x_{5}',x_{v,x_{5}'},v,x_{x_{1}',v},x_{1}'\}$, either $w$ is adjacent to a vertex coloured $f(t)$ that is not $t$, or $t$ is adjacent to a vertex coloured $f(w)$ that is not $w$. If not, we exchange the two colours and extend the colouring using the same ideas as in the above case analysis. Following a similar argument as in Lemma \ref{pathoflength3kempe}, we now see that at least one of the desired outcomes must follow. 
\end{proof}

\section{A basic counting argument to finish}
\label{finishedhard}

In this section we prove Theorem \ref{maintheorem}. We assume that all components of the Gallai Tree are isomorphic to paths. Let $P$ be a path component in the Gallai Tree. Recall that $\deg_{3}(v)$ denotes the number of neighbours of $v$ which have degree $3$. Assign to each vertex $v$ a charge of $\deg(v)$. Consider the discharging rule where each vertex $v$ with $\deg(v) \geq 4$ sends $\frac{\deg(v) - 3.4}{\deg_{3}(v)}$ charge to each of its neighbours of degree $3$. Let $\text{ch}(v)$ denote the charge of each vertex after performing the discharging rule. 

\begin{obs}
If $\deg(v) \geq 4$, then $\text{ch}(v) = 3.4$.
\end{obs}

\begin{proof}
We have that $\text{ch}(v) = \deg(v) - \deg_{3}(v)(\frac{\deg(v)-3.4}{\deg_{3}(v)}) = 3.4$.
\end{proof}

Given a component $P$ of the Gallai Tree, we let $\text{ch}(P) = \sum_{v \in V(P)} \text{ch}(v)$. Observe that if for every component of the Gallai Tree $P$, we have $\text{ch}(P) \geq 3.4v(P)$, then Theorem \ref{maintheorem} follows.
To see this we have $2e(G) = \sum_{v \in V(G)}\deg(v) = \sum_{v \in V(G)} \text{ch}(v) \geq 3.4v(G)$, and hence $e(G) \geq \frac{17v(G)}{10}$. We will say a component $P$ of the Gallai Tree is \textit{safe} if $\text{ch}(P) \geq 3.4v(P)$. Thus we devote the rest of the section to showing that all components of the Gallai Tree are safe. 

\begin{prop}
Let $P$ be an isolated vertex in the Gallai Tree. Then $P$ is safe. 
\end{prop}

\begin{proof}
Let $v$ be the isolated vertex in $P$. Then all neighbours of $v$ have degree at least $4$. Hence $\text{ch}(v) \geq 3 + 3(\frac{4-3.4}{4}) \geq 3.45$, and thus $v$ is safe. 
\end{proof}

\begin{prop}
Let $P$ be isomorphic to an edge in the Gallai Tree. Then $P$ is safe.
\end{prop}

\begin{proof}
Let $v$ be a vertex in $V(P)$.  Let $x,y,z$ be the neighbours of $v$, and without loss of generality let $z$ be the neighbour of $v$ with degree $3$. First suppose that $x,y,z$ are an independent set. Then by Lemma \ref{doubleneighbourhoodreconfig} there are distinct vertices which are not in $\{x,y,z,v\}$, $x_{y,z}$, $x_{x,z}$, $x_{x,y}$ which are adjacent to $y$ and $z$,  $x$ and $z$ and $x$ and $y$ respectively. Then as $P$ is isomorphic to an edge, $\deg(x_{x,z}) \geq 4$ and $\deg(x_{x,y}) \geq 4$. Hence for all $t \in \{x,y,x_{x,z},x_{y,z}\}$, $\deg_{3}(t) \leq \deg(t)-1$. Therefore $\text{ch}(v) \geq 3 + 2(.2) = 3.4$ and $\text{ch}(z) \geq 3 + 2 (.2) =3.4$. Therefore it follows in this case that $P$ is safe. Observe that if $x,y,z$ is not an independent set but only $xy \in E(G)$, then the above argument still shows that $P$ is safe, as we never considered $x_{x,y}$ (in fact, the edge $xy$ improves the situation). 

Therefore we can assume that $z$ is adjacent to at least one of $x$ or $y$. Note that $z$ cannot be adjacent to both $x$ and $y$, as otherwise we contradict Proposition \ref{closedneighbourhoods}.

So without loss of generality suppose that $yz \in E(G)$. Then $z$ and $v$ share a common neighbour, $y$, and thus by Lemma \ref{keywheelcutlemma} $yx \in E(G)$. Additionally, the neighbour of $z$ which is not $y$ or $v$, say $z'$ is also adjacent to $y$. Further $\deg(z') \geq 4$. Therefore for all $t \in \{x,y,z'\}$, we have $\deg_{3}(t) \leq \deg(t) -1$, and $\deg_{3}(y) \leq \deg(y) -2$. Hence  $\text{ch}(v) \geq 3 + .3 +.2 = 3.5$ and $\text{ch}(z) \geq 3 + .3 + .2 = 3.5$. Therefore in this case $P$ is safe and thus the proposition follows. 
\end{proof}

\begin{prop}
Let $P$ be isomorphic to a path of length $2$ in the Gallai Tree. Then $P$ is safe. 
\end{prop}

\begin{proof}
Let $P= x,y,z$ be the path of length $2$ in the Gallai Tree. Let $x',x'',y',z',z''$ be the vertices adjacent to $x,y,z$ respectively not on $P$. Then by Lemma \ref{reconfiguringpathsoflength3} up to relabelling the vertices, we have that $y'$ is adjacent to $x''$ and $z''$, and $x' = z'$.  If $x' = y'$, then $\deg(y') \geq 5$, and $\deg_{3}(y') \leq \deg(y')-2$, and hence $\frac{\deg(y') -3.4}{\deg_{3}(y')} \geq \frac{8}{15}$. Thus $\text{ch}(P) \geq 9 + 3(\frac{8}{15}) + .4 = 11 \geq 10.2$, and hence $P$ is safe in this case. 

Therefore $x' \neq y'$. If $x'$ is adjacent to both $x''$ and $z''$, then for any $t \in \{x',x'',y',z''\}$, we have $\deg(t) \geq 4$ and $\deg_{3}(t) \leq \deg(t) -2$. Therefore $\text{ch}(P) \geq 9 + 5(.3) = 10.5 \geq 10.2$, and $P$ is safe in this case. 

Therefore $x'$ is adjacent to at most one of $x''$ or $z''$. Suppose $x'$ is not adjacent to $z''$ but is adjacent to $x''$. Then there is a vertex not in $P$, $x_{x',z''}$ which is adjacent to both $x'$ and $z''$. Then for $t \in \{y,x''\}$, we have $\deg(t) \geq 4$ and $\deg_{3}(t) \leq \deg(t) -2$. For $w \in \{x',z''\}$, we have $\deg(w) \geq 4$ and $\deg_{3}(t) \leq \deg(t) -1$. Hence we have $\text{ch}(P) \geq 9 + 1.2 = 10.2$ and hence $P$ is safe. 

 Thus by symmetry we may assume that $x'$ is not adjacent to either $x''$ or $z''$. Then there are distinct vertices not in $P$, say $x_{x',x''}$ and $x_{x',z''}$, which are adjacent to $x'$ and $x''$, and $x'$ and $z''$ respectively (and further $x_{x',x''}, x_{x',z''} \not \in \{x'',y',z'',x'\}$). 
 
 If $\deg(x') \geq 5$, then observe that 
$\text{ch}(P) \geq 9 + .64 + .4 + .3 = 10.34 > 10.2$ and hence $P$ is safe. 

If one of $x''$ or $z''$ has degree $5$, then we have that
$\text{ch}(P) \geq 9 + .4 + .3 + 2(.15) + .2 = 10.2$, and hence $P$ is safe. 

If $\deg(y) \geq 5$, then $\text{ch}(P) \geq 9 + .2 + .3 + .2 + \frac{8}{15} \geq 10.23 > 10.2$.

If either of $\deg(x_{x',z''}) \geq 4$ or $\deg(x_{x',x''}) \geq 4$, then again $\text{ch}(P) \geq 9 + 2(.3) + 3(.2) = 10.2$, and hence $P$ is safe. 

By Lemma \ref{pathoflength3kempe}, at least one of the above cases occurs, and hence $P$ is safe. 
\end{proof}

\begin{prop}
Let $P$ be isomorphic to a path of length $3$ in the Gallai Tree. Then $P$ is safe.
\end{prop}

\begin{proof}
Let $P = x_{1},x_{2},x_{3},x_{4}$. By Corollary \ref{alternatingpath}, up to relabelling the vertices, there are vertices $x_{1}',x_{4}'$ adjacent to $x_{1}$ and $x_{4}$ respectively, vertices $u$ and $v$ (distinct from $x_{1}'$ and $x_{4}')$), such that $u$ is adjacent to $x_{4}',x_{3},x_{1}$, and $v$ is adjacent to $x_{1}',x_{2}$ and $x_{4}$. If $u = v$, then $\deg(u) \geq 6$, and in this case it follows that $\text{ch}(P) \geq 14.6 >13.6$ and hence $P$ is safe. Therefore we can assume that $u \neq v$. 

If $ux_{1}' \in E(G)$, then as $\deg_{3}(u) \leq \deg(u) -2$, and $\deg_{3}(x_{1}') \leq \deg(x_{1}')-2$, it follows that $\text{ch}(P) \geq 13.6$ and hence $P$ is safe. An analagous argument holds if $vx_{4}' \in E(G)$.

Hence we can assume that $ux_{1}' \not \in E(G)$ and $vx_{4}' \not \in E(G)$. Then there are distinct vertices $x_{u,x_{1}'}$ and $x_{v,x_{4}'}$ not in $P$ that are adjacent to $u$ and $x_{1}'$ and $v$ and $x_{4}'$ respectively. 

Therefore, we have a subgraph isomorphic to the graph in Lemma \ref{pathoflength4kempe}. If either $\deg(v) \geq 5$ or $\deg(u) \geq 5$, then $\text{ch}(P) \geq 14 >13.6$ and hence $P$ is safe.  

If $\deg(x_{1}') \geq 5$, and $\deg(x_{4}') \geq 5$, then $\text{ch}(P) \geq 12 + .4 + .8 + .4 =  13.6$ and hence $P$ is safe in this case. 

If $\deg(x_{1}') \geq 5$ and $\deg(x_{x'_{4},u}) \geq 4$ then $\text{ch}(P) \geq 12 + .4 + .4 +1.2 = 14 > 13.6$. 

If $\deg(x_{1}') \geq 5$ and $\deg(x_{x'_{1},v}) \geq 4$, then 
$\text{ch}(P) \geq 12 + 1.73 > 13.6$ and hence $P$ is safe in this case.

If $\deg(x_{x'_{1},v}) \geq 4$ and $\deg(x_{x'_{4},u}) \geq 4$, then $\text{ch}(P) \geq 12 + 6(.3) = 13.8 > 13.6$ and hence $P$ is safe in this case. 

One of these cases occur by Lemma \ref{pathoflength4kempe}, and hence we get that $P$ is safe in all cases. 
\end{proof}

\begin{prop}
Let $P$ be isomorphic to a path of length $4$ in the Gallai Tree. Then $P$ is safe.
\end{prop}

\begin{proof}
Let $P = x_{1},x_{2},x_{3},x_{4},x_{5}$. We apply Corollary \ref{alternatingpath} to $P$. Then from Corollary \ref{alternatingpath} there are vertices $x_{1}'$ and $x_{5}'$ not in $P$ where $x_{1}'$ is adjacent to $x_{1}$ and $x_{5}'$ is adjacent to $x_{5}$ and given a bipartition $(A,B)$ of $P \cup \{x_{1}',x_{5}'\}$, there are vertices $u$ and $v$ so that $u$ is adjacent to all vertices in $A$ and $v$ is adjacent to all vertices in $B$. If $u = v$, then $\deg(u) \geq 7$, and hence $\text{ch}(P) \geq 15 + 5(.72) +2(.2) = 19 > 17$. 

Therefore $u \neq v$. Without loss of generality we can assume that $x_{1}' \in A$, and hence $x_{5}' \in A$. First suppose that $x_{1}'$ is adjacent to $v$. Then either $\deg(v) \geq 5$, or $x_{1}'x_{5}' \in E(G)$. If $x_{1}'x_{5}' \in E(G)$, then $\text{ch}(P) \geq 15 + 3(.2) + .6 + 2(.3) + .3 = 17.1 > 17$ and hence $P$ is safe in this case. If $\deg(v) \geq 5$, then $\text{ch}(P) \geq 15 + 3(.4) + .3 + .2 + 2(.3) =17.3 >17$ and hence $P$ is safe in this case. 

Therefore we can assume that $v$ is not adjacent to $x_{1}'$, and by symmetry we can assume that $v$ is not adjacent to $x_{5}'$. Therefore there are vertices $x_{x_{1}',v}$ and $x_{x_{5}',v}$ not on $P$ which are adjacent to $x_{1}'$ and $v$, and $x_{5}'$ and $v$ respectively. 

If $\deg(v) \geq 6$, then $\text{ch}(P) \geq 15 + 4(\frac{13}{30}) + .4 + .6 \geq 17 + \frac{11}{15}$, and hence $P$ is safe in this case. 

If $\deg(u) \geq 5$, then $\text{ch}(P) \geq 15 + 2(.53) + 3(.32) + .4 = 17.42 > 17$. Thus $P$ is safe in this case. 

Therefore we have a subgraph as in Lemma \ref{pathoflength5kempe}. If $\deg(x_{x_{1}',v}) \geq 4$ then $\text{ch}(P) \geq 15 + 3(.32) + .9 + .2 = 17.06 >17$, and hence $P$ is safe in this case. Similarly if $\deg(x_{x_{5}',v}) \geq 5$, then $P$ is safe. If $\deg(x_{1}') \geq 5$, then $\text{ch}(P) \geq 3(.32) + .4 + .6 + .2 = 17.16 > 17$, and hence $P$ is safe in this case. Similarly, if $\deg(x_{5}') \geq 5$ then $P$ is safe. 
\end{proof}

\begin{prop}
All path components in the Gallai Tree of length at least $5$ are safe. 
\end{prop}

\begin{proof}
Let $P$ be a path of length at least $5$ in the Gallai Tree, with endpoints $u$ and $v$. We apply Corollary \ref{alternatingpath}. Then there are vertices $u'$ and $v'$ adjacent to $u$ and $v$ respectively such that for a bipartition $(A,B)$ of $P \cup \{u',v'\}$,  there are vertices $x$ and $y$ such that $x$ is adjacent to all vertices in $A$ and $y$ is adjacent to all vertices in $B$. Further, for any $w \in \{u',v'\}$ and $t \in \{x,y\}$, if $wt \not \in E(G)$, then there is a vertex $x_{w,t}$ which is not in $P$, and adjacent to both $w$ and $t$. As $P$ has length at least five, the degree of $x$ and $y$ is at least $5$.

 Observe as the length of the path is at least five, this implies that the degree of $x$ and $y$ is at least five.  If $x = y$, then $\deg(x) \geq 8$, $\frac{\deg(x) -3.4}{\deg_{3}(x)} \geq \frac{23}{30}$, and hence $\text{ch}(P) \geq 3v(P) + \frac{23}{30}v(P) > 3.4v(P)$. It follows that $P$ is safe. 
 
 Thus $x \neq y$. If $x$ and $y$ have degree at least $6$, then for $t \in \{x,y\}$, we have $\frac{\deg(t)-3.4}{\deg_{3}(t)} \geq \frac{13}{30}$. Thus $\text{ch}(P) \geq 3v(P) + \frac{13}{30}v(P) > 3.4v(P)$ and it follows that $P$ is safe. 
 
 Thus at least one of $x$ or $y$ has degree $5$. Thus the length of $P$ is at most $6$. If the length of $P$ is exactly $6$, then the vertex of degree $5$ is adjacent to both $u'$ and $v'$, and thus is adjacent to at most $3$ vertices of degree three. In this case we have that $\text{ch}(P) \geq 3v(P) + 3(\frac{8}{15}) + 4(\frac{13}{30}) > 3.4v(P)$. 
 
  Thus the last case to consider is when the length of $P$ is exactly $5$. Observe that either the degree of $x$ and $y$ is greater than $6$, or it is $5$ and the number of neighbours of degree $3$ is at most $4$ (since $x$ is adjacent to at least one of $u'$ and $v'$, and $y$ is adjacent to at least one of $u'$ and $v'$). As $\frac{5-3.4}{4} = .4$, it again follows that $P$ is safe. 
\end{proof}

Thus every component of the Gallai Tree is safe, and Theorem \ref{maintheorem} follows. We remark that certain aspects of this proof are easily improvable, but the bottleneck occurs in dealing components of the Gallai Tree which are isomorphic to a path of length $2$, and it is unclear to me how to improve the bound significantly in this case.  It would be nice to be able to use the potential method instead of the above techinques to obtain a better bound, but I did not see how to do so effectively.

\bibliographystyle{plain}
\bibliography{Potentialbib}

\end{document}